\documentclass[final,leqno,onefignum,onetabnum]{siamltex1213}

\usepackage{tikz}
\usetikzlibrary{arrows,chains,matrix,positioning,scopes}
\makeatletter
\tikzset{join/.code=\tikzset{after node path={%
\ifx\tikzchainprevious\pgfutil@empty\else(\tikzchainprevious)%
edge[every join]#1(\tikzchaincurrent)\fi}}}
\makeatother
\tikzset{>=stealth',every on chain/.append style={join},
         every join/.style={->}}
\tikzstyle{labeled}=[execute at begin node=$\scriptstyle,
   execute at end node=$]

\usepackage[hyphenbreaks]{breakurl}
\usepackage{cite}
\usepackage{amsmath, amssymb, amsfonts}
\usepackage{array}

\allowdisplaybreaks
\newtheorem{example}{\textit{Example}}[section]

\usepackage{rotating}
\usepackage{tablefootnote}

\usepackage{algorithm2e}
\usepackage[noend]{algpseudocode}

\usepackage[multiple]{footmisc}

\title{Algorithm for Optimization and Interpolation based on Hyponormality}

\author{C\'edric Josz\footnotemark[1] }

\usepackage{soul}

\begin{document}
\maketitle

\renewcommand{\thefootnote}{\fnsymbol{footnote}}

\footnotetext[1]{Laboratory for Analysis and Architecture of Systems (LAAS), French National Center for Scientific Research (CNRS), 7, avenue du Colonel Roche, Toulouse, 31000, France (\email{cedric.josz@gmail.com}). The research was funded by the European Research Council (ERC) under the European Union's Horizon 2020 research and innovation program (grant agreement 666981 TAMING).}

\renewcommand{\thefootnote}{\arabic{footnote}}

\slugger{mms}{xxxx}{xx}{x}{x--x}%slugger should be set to mms, siap, sicomp, sicon, sidma, sima, simax, sinum, siopt, sisc, or sirev

\begin{abstract}
On one hand, consider the problem of finding global solutions to a polynomial optimization problem and, on the other hand, consider the problem of interpolating a set of points with a complex exponential function. This paper proposes a single algorithm to address both problems. It draws on the notion of hyponormality in operator theory. Concerning optimization, it seems to be the first algorithm that is capable of extracting global solutions from a polynomial optimization problem where the variables and data are complex numbers. It also applies to real polynomial optimization, a special case of complex polynomial optimization, and thus extends the work of Henrion and Lasserre implemented in GloptiPoly. Concerning interpolation, the algorithm provides an alternative to Prony's method based on the Autonne-Takagi factorization and it avoids solving a Vandermonde system. The algorithm and its proof are based exclusively on linear algebra. They are devoid of notions from algebraic geometry, contrary to existing methods for interpolation. The algorithm is tested on a series of examples, each illustrating a different facet of the approach. One of the examples demonstrates that hyponormality can be enforced numerically to strenghten a convex relaxation and to force its solution to have rank one.
\end{abstract}



\pagestyle{myheadings}
\thispagestyle{plain}
\markboth{C\'EDRIC JOSZ}{Algorithm for Optimization and Interpolation based on Hyponormality}
\section{Introduction}
Consider the problem of finding global solutions to the following complex polynomial optimization problem
\begin{equation}
\label{eq:pop}
\begin{array}{ll}
\inf_{z \in \mathbb{C}^n} ~ & f(z) ~:= \sum\limits_{\alpha,\beta} f_{\alpha,\beta} \bar{z}^\alpha z^\beta \\[1em]
\mathrm{s.t.} & g_i(z) := \sum\limits_{\alpha,\beta} g_{i,\alpha,\beta} \bar{z}^\alpha z^\beta \geqslant 0, \quad i=1,\ldots,m
\end{array}
\end{equation}
where we use the multi-index notation $z^\alpha := z_1^{\alpha_1} \cdots z_n^{\alpha_n}$ for $z \in {\mathbb C}^n$,
$\alpha \in {\mathbb N}^n$, and $\bar{z}$ stands for the conjugate of $z$. As usual, $\mathbb{C}$ denotes the set of complex numbers and $\mathbb{R}$ will denote the set of real numbers. The functions $f, g_1, \ldots, g_m$ are real-valued polynomials so that in the above sums only a finite number of coefficients $f_{\alpha,\beta}$ and $g_{i,\alpha,\beta}$ are nonzero and they satisfy $\overline{f_{\alpha,\beta}} = f_{\alpha,\beta}$ and $\overline{g_{i,\alpha,\beta}} = g_{i,\alpha,\beta}$. The feasibility set is defined as
\begin{equation}
K:=\{z \in {\mathbb C}^n \: :\: g_i(z) \geqslant 0, \: i=1,\ldots,m\}.
\end{equation}
We define its degree to be 
\begin{equation}
d_K := \max \{1,k_1,\hdots,k_m\}
\end{equation}
where $k_i := \max \{ |\alpha|,|\beta| ~\text{s.t.}~ g_{i,\alpha,\beta} \neq 0 \}$ is the maximal degree $|\alpha|:= \sum_{k=1}^n \alpha_k$ in either the conjugate or non-conjugate powers of the polynomial $g_i$. Note that this is different from the degree of the polynomial, which is related to the sum of the conjugate and non-conjugate powers, i.e. $\text{deg}(g_i) := \max \{ |\alpha|+|\beta| ~\text{s.t.}~ g_{i,\alpha,\beta} \neq 0 \}$.

To solve (\ref{eq:pop}), Molzahn and I proposed \cite{josz-molzahn-2015} proposed a semidefinite programming relaxation hierarchy in complex numbers which generalizes Lasserre's hierarchy \cite{lasserre-2000,lasserre-2001,parrilo-2000b,parrilo-2003} and relies on the recent results \cite{angelo-2008,putinar-2013}. It consists in the following primal-dual problems 
\begin{equation}
\begin{array}{c}
\inf_y ~ L_y(f) ~~~ \text{s.t.} ~~~  M_d(y) \succcurlyeq 0 ~\text{and}~ M_{d-k_i}(g_iy) \succcurlyeq 0, ~ i = 1, \hdots, m, \\\\
\sup_{\lambda \in \mathbb{R}, \sigma_j \in \Sigma_d[z]} ~ \lambda ~~~ \text{s.t.} ~~~ f - \lambda = \sigma_0 + \sigma_1 g_1 + \hdots + \sigma_m g_m
\end{array}
\end{equation}
where $\succcurlyeq 0$ stands for positive semidefinite and the integer $d$ is the \textit{truncation order}. The \textit{moment} matrix is defined by 
\begin{equation} 
M_d(y) := (y_{\alpha,\beta})_{|\alpha|,|\beta|\leqslant d}
\end{equation}
and the \textit{localizing} matrix is defined by 
\begin{equation}
M_{d-k_i}(g_iy) := \left(\sum_{\gamma,\delta} g_{i,\gamma,\delta} y_{\alpha,\beta}\right)_{|\alpha|,|\beta|\leqslant d-k_i}.
\end{equation}
A polynomial $\sigma(z) = \sum_{|\alpha|,|\beta|\leqslant d} \sigma_{\alpha,\beta} \bar{z}^\alpha z^\beta$ is a \textit{Hermitian sum of squares}, i.e. it belongs to $\Sigma_d[z]$, if it is of the form
\begin{equation}
\sigma(z) = \sum_k \left| \sum_{|\alpha|\leqslant d} p_{k,\alpha} z^\alpha \right|^2  ~~\text{where}~~ p_{k,\alpha} \in \mathbb{C}.
\end{equation}
This is equivalent to $(\sigma_{\alpha,\beta})_{|\alpha|,|\beta|\leqslant d} \succcurlyeq 0$ where $\succcurlyeq$ stands for positive semidefinite (see \cite{josz-molzahn-2015} for an explanation). In the above equation, $|\cdot|$ stands for modulus of a complex number.

If one were to convert a complex polynomial optimization problem to real numbers and apply Lasserre's hierarchy, the moment matrix would be $2^d$ times bigger asymptotically (for a large number of variables), hence the advantage of the complex hierarchy we've just described. This comes at the cost of a potentially lower optimal value at a given truncation order $d$. Significant numerical advantages can be seen on the optimal power flow in electrical engineering \cite{carpentier-1962,mh_sparse_msdp,josz-molzahn-2015} on instances with several thousand variables and constraints.

Global convergence is guaranteed in the presence of a sphere constraint, i.e. $|z_1|^2 + \hdots + |z_n|^2 = R^2$. At the cost of an additional variable, any complex polynomial optimization problem with compact feasible set can be solved by this approach, as explained in \cite{josz-molzahn-2015}. In that work,
conditions for extracting global minimizers were given, but a general procedure for extracting them was left for future work. One of the objectives of this paper is to fill this gap. Specifically, we propose an algorithm to extract an atomic measure $\mu$ from the truncated data $M_d(y)$, i.e. that satisfies
\begin{equation}
y_{\alpha,\beta} = \int_K \bar{z}^\alpha z^\beta d\mu, ~~~ \forall |\alpha|,|\beta| \leqslant d.
\end{equation}
The atoms are then global solutions to the polynomial optimization problem. There exists no method in the present literature that achieves this to the best of our knowledge. Our algorithm also applies to real polynomial optimization, for which a method already exists \cite{henrion2005} and was implemented in Gloptipoly \cite{henrion2009}. A variant to that method was later proposed in \cite{laurent2009}. We next consider a seemingly unrelated problem for which the same algorithm applies.

%Given some observed data
%Consider the damped sinusoidal function
%\begin{equation}
%\begin{array}{rccl}
%f : & \mathbb{R}^n & \longrightarrow & \mathbb{R} \\
%& t\hphantom{^n} & \longmapsto & \sum_{k=1}^p A_k \exp(\sum_{i=1}^n \sigma_{ki} t_i)\cos( \sum_{i=1}^n w_{ki} t_i + \phi_k)
%\end{array}
%\end{equation}
%\begin{equation}
%f(t) = \Re \left\{ \sum_{k=1}^p w_k \exp(\sum_{i=1}^n f_{ki} t_i) \right\}
%\end{equation}
Consider the following sum of complex exponential functions
\begin{equation}
\label{eq:exp}
\begin{array}{rccl}
f : & \mathbb{C}^n & \longrightarrow & \mathbb{C} \\
& z\hphantom{^n} & \longmapsto & \sum\limits_{k=1}^d w_k \exp\left(\sum\limits_{i=1}^n f_{ik} z_i\right)
\end{array}
\end{equation}
composed of weights $w_1, \hdots , w_p \in \mathbb{C}^n$ and frequencies $f_1,\hdots,f_d \in \mathbb{C}^n$ (using the shorthand $f_k = ( f_{1k} , \hdots , f_{nk})^T$ where $(\cdot)^T$ stands for transpose). Say we want to interpolate a set of imposed values $(y_{\alpha})_{|\alpha| \leqslant 2d}$ with such a function
\begin{equation}
\label{eq:interpolation}
f(\alpha) = y_\alpha~,~~~ \forall |\alpha|\leqslant 2d.
\end{equation}
In other words, the problem consists in computing weights and frequencies that match the interpolation values. As is well known (e.g.,\cite{kunis2016}), these values satisfy
\begin{equation}
f(\alpha) = \sum_{k=1}^d w_k \exp\left(\sum_{i=1}^n f_{ki} \alpha_i\right) = \sum_{k=1}^d w_k (\exp(f_k))^{\alpha} = \int_{\mathbb{C}^n} z^\alpha d\nu
\end{equation}
where 
\begin{equation}
\nu := \sum_{k=1}^d w_k \delta_{\exp(f_k)}.
\end{equation}
We use the notation $\exp(f_k) := ( \exp(f_{1k}), \hdots, \exp(f_{nk}) )^T$ and $\delta_z$ stands for the Dirac measure at $z \in \mathbb{C}^n$. The interpolation values are thus the moments on $\mathbb{C}^n$ of the measure $\nu$. In this setting, we can consider the following moment matrix 
\begin{equation}
\label{eq:hankel}
\mathcal{H}_d(y) := (y_{\alpha+\beta})_{|\alpha|,|\beta|\leqslant d}
\end{equation}
which is a complex Hankel matrix (i.e. $y_{\alpha,\beta} = y_{\gamma,\delta}$ for all $|\alpha|,|\beta|,|\gamma|,|\delta| \leqslant d$ such that $\alpha+\beta=\gamma+\delta$). Our algorithm extracts the sought measure $\nu$ from this matrix by using the Autonne-Takagi factorization \cite{takagi,autonne}, also known as symmetric singular value decomposition, which applies to complex symmetric matrices. To the best of our knowledge, this factorization has not been used in this context. For a thorough survey on the applications of complex symmetry, see \cite{garcia2014}; for recent development on the Autonne-Takagi factorization, see \cite{horn2012}. Another feature of our algorithm is that it avoids solving a Vandermonde linear system, as in the recent preprint \cite{harmouch2017}. Existing methods were initiated by Baron Gaspard Riche de Prony in 1795 \cite{prony1795}, and have been subject to formidable developments
\cite{beylkin2005,pereyra2010,sauer2016,andersson2010,mourrain2016,kunis2016bis} and applications \cite{backstrom2013,golub2003,potts2013,potts2010,roy1990,swindlehurst1992}. One application is to recover a damped sinusoidal function of real variable $t$ of the form
\begin{equation}
f(t) = \sum\limits_{k=1}^d A_k \exp(\sigma_{k} t)\cos( w_{k} t + \phi_k)
%\begin{array}{rccl}
%f : & \mathbb{R}^n & \longrightarrow & \mathbb{R} \\
%& t\hphantom{^n} & \longmapsto & \sum_{k=1}^d A_k \exp(\sum_{i=1}^n \sigma_{ki} t_i)\cos( \sum_{i=1}^n w_{ki} t_i + \phi_k)
%\end{array}
\end{equation}
from a small number of its evaluations, where
\begin{itemize}
\item $A_k$: amplitude
\item $\sigma_k$: damping
\item $w_k$: angular frequency
\item $\phi_k$: phase shift.
\end{itemize}
Indeed, such a function is special case of a complex exponential function presented above since $\cos(w_{k} t + \phi_k)= 1/2 \exp(i w_{k} t + i\phi_k) + 1/2 \exp(-iw_{k} t - i\phi_k) $. We briefly recall Prony's method in the univariate setting of \eqref{eq:exp}-\eqref{eq:interpolation} above. It consists in defining the polynomial
\begin{equation}
p(z) := \sum_{|\alpha| \leqslant d} p_{\alpha} z^\alpha := (z - \exp(f_1) ) \hdots (z-\exp(f_d)) 
\end{equation}
which, for all $|\alpha| \leqslant d$, satisfies 
\begin{equation}
\begin{array}{rcl}
\sum\limits_{|\beta| \leqslant d} f(\alpha+\beta) ~ p_{\beta} & = & \sum\limits_{|\beta| \leqslant d} \sum\limits_{k=1}^d w_k \exp(f_k(\alpha+\beta)) ~ p_\beta \\[1.15em]
 & = & \sum\limits_{k=1}^d w_k \exp(f_k \alpha) ~ \sum\limits_{|\beta| \leqslant d}  p_\beta \exp(f_k)^\beta \\[.5em]
 & = & \sum\limits_{k=1}^d w_k \exp(f_k \alpha) ~ p(\exp(f_k)) = 0
\end{array}
\end{equation}
so that its coefficients lie in the Kernel of the Hankel matrix $\mathcal{H}_d(y)$ in \eqref{eq:hankel}. The Kernel is unidimensional and $p_d = 1$, so the coefficients are uniquely determined. The frequencies of the interpolation function are given by the its roots, and the weights can then be deduced by the Vandermonde system
\begin{equation}
\begin{pmatrix}
\hphantom{^{d}}1\hphantom{^{d-1}} & \hdots & 1\hphantom{^{d-1}} \\
\hphantom{^{d}}\exp(f_1)\hphantom{^{d-1}} & \hdots & \exp(f_d)\hphantom{^{d-1}} \\
\vdots & & \vdots \\
\hphantom{^{d}}\exp(f_1)^{d-1} & \hdots & \exp(f_d)^{d-1}
\end{pmatrix}
\begin{pmatrix}
w_1 \\[.25cm]
\vdots \\[.25cm]
w_d
\end{pmatrix}
=
\begin{pmatrix}
f(0) \\[.25cm]
\vdots \\[.25cm]
f(d-1)
\end{pmatrix}.
\end{equation}

This paper is organized as follows. Section \ref{sec:Contributions} summarizes the contributions of this paper. Section \ref{sec:Joint hyponormality} provides some background on hyponormality and compares the notion in infinite and finite dimensions. Section \ref{sec:Truncated moment problem} reviews a result on the moment problem in complex numbers and provides a new result for Toeplitz matrices. The same machinery is then applied to the truncated moment problem which arises in exponential interpolation. The constructive proofs of Section \ref{sec:Truncated moment problem} yield an algorithm for optimization and interpolation in Section \ref{sec:Algorithm}. It is illustrated on various examples in Section \ref{sec:Numerical experiments}. Finally, Section \ref{sec:Conclusion} concludes our work.

\section{Contributions}
\label{sec:Contributions}
The main contribution of this paper is to apply hyponormality to polynomial optimization and exponential interpolation. This has not been considered in past literature to the best of our knowledge. We obtain the following new results using hyponormality.
\begin{enumerate}
\item We propose a procedure to extract global solutions from a polynomial optimization problem when using the moment/sum-of-squares hierarchy in complex numbers (algorithm in Section \ref{sec:Algorithm}).
\item We propose a variant to Prony's method based on the Autonne-Takagi factorization using the same algorithm in Section \ref{sec:Algorithm}, thereby unifying two \textit{a priori} unrelated problems: optimization and interpolation.
\item We propose a moment/sum-of-squares hierarchy for complex polynomial optimization in which hyponormality is enforced via convex constraints and illustrate it on an example (Example \ref{eg:ellipsebis} in Section \ref{sec:Numerical experiments}).
\item We propose a solution to the truncated moment problem on a semi-algebraic set when the data forms a Toeplitz matrix (Theorem \ref{th:toeplitz} in Section \ref{sec:Truncated moment problem}). We find that Toeplitz and Hankel matrices play analogeous role with respect to the moment problem. 
%Toeplitz matrices are relevant for optimization over the torus (Example \ref{eg:toeplitz}), while Hankel matrices are relevant for optimization over the reals (Example \ref{eg:real}).
\item We analyse the different properties of shift operators associated to truncated data (eg. unitary, symmetric) and relate them to applications in optimization and interpolation (examples 1 through 7 in Section \ref{sec:Numerical experiments}).
\end{enumerate}

\section{Joint hyponormality}
\label{sec:Joint hyponormality}
We recall some fundamental notions in operator theory in order to discuss joint hyponormality. Let $\mathcal{B}(\mathcal{H})$ denote the set of linear bounded operators acting on a complex Hilbert space $(\mathcal{H}, \langle \cdot, \cdot \rangle)$. We convene that the inner product is conjugate-linear in its first variable, and linear in its second. A notion of positivity can be defined for an element $T$ of $\mathcal{B}(\mathcal{H})$, which is that $\langle u , T u \rangle \geqslant 0$ for all $u \in \mathcal{H}$, and which we denote $T \succcurlyeq 0$. In addition, the commutator of $A,B \in \mathcal{B}(\mathcal{H})$ is defined as $[A,B]:= AB - BA$. Finally, let $A^*$ denote the adjoint of $A \in \mathcal{B}(\mathcal{H})$. Given these notations, an operator $T \in \mathcal{B}(\mathcal{H})$ is said to be $\hdots$
\begin{itemize}
\item \textit{normal} if $[T^*,T] = T^*T - TT^* = 0$;
\item \textit{subnormal} if it can be extended to a normal operator $N$ on a larger Hilbert space $\mathcal{K}$;
\item \textit{hyponormal} if $[T^*,T] = T^*T - TT^* \succcurlyeq 0$.
\end{itemize}
The following implications hold
$$ \text{normal} ~~~ \Longrightarrow ~~~ \text{subnormal} ~~~ \Longrightarrow ~~~ \text{hyponormal} $$
The first implication is obvious. Concerning the second, let $P$ denote the orthogonal projection of $\mathcal{K}$ on $\mathcal{H}$. Then, for all $u \in \mathcal{H}$, we have $Tu = NPu$ and thus, $T^* = (NP)^* = P^* N^* = P N^*$ as projections are self-adjoint. Following \cite{curto-2010}, we have
\begin{equation}
\begin{array}{rcl}
\langle u , [T^*,T] u \rangle & = & \| T u \|^2 - \|T^*u\|^2 \\
 & = &  \| N u \|^2 - \|PN^*u\|^2 \\
 & = &  \| N^* u \|^2 - \|P(N^*u)\|^2 \geqslant 0 \\
\end{array}
\end{equation}
where $\|\cdot\|$ stands for the norm induced by the inner product. The notions of subnormality and hyponormality were introduced by Halmos \cite{halmos1950} in 1950 when studying the unilateral shift operator, that is 
\begin{equation}
\begin{array}{rccc}
T : & l^2(\mathbb{N}) &\longmapsto & l^2(\mathbb{N}) \\
& (u_0,u_1,u_2, \hdots ) & \longrightarrow & (0,u_0,u_1,\hdots)
\end{array}
\end{equation}
where $l^2(\mathbb{N})$ is the Hilbert space of square-summable sequences of complex numbers indexed by the natural numbers $\mathbb{N}$.
Its adjoint is 
\begin{equation}
\begin{array}{rccc}
T^* : & l^2(\mathbb{N}) &\longmapsto & l^2(\mathbb{N}) \\
& (u_0,u_1,u_2, \hdots ) & \longrightarrow & (u_1,u_2,u_3,\hdots)
\end{array}
\end{equation}
It is not normal since for all $u\in l^2(\mathbb{N})$, we have $(T^*T - TT^*)u = (u_0,0,\hdots,0,\hdots)$. However, it admits the following normal extension
\begin{equation}
\begin{array}{rccc}
N : & l^2(\mathbb{Z}) &\longmapsto & l^2(\mathbb{Z}) \\
& (\hdots, u_{-2},u_{-1},u_0,u_1,u_2, \hdots ) & \longrightarrow & (\hdots,u_{-3},u_{-2},u_{-1},u_0,u_1,\hdots)
\end{array}
\end{equation}
where $l^2(\mathbb{Z})$ is the Hilbert space of square-summable sequences indexed by the integers $\mathbb{Z}$.
By definition, $T$ is subnormal, and, as a consequence, it is hyponormal. We can check that for all $u\in l^2(\mathbb{N})$, we have $\langle u , [T^*,T] u \rangle = |u_0|^2 \geqslant 0$. Hyponormality has been subject to various developments since it was first introduced, see for example \cite{conway1988,curto2001,douglas1989,douglas1992,farenick1995,mccullough1989,xia1983,curto1987,curto1988}.

In this paper, we will define a shift operator for each variable (optimization variable or variable of interpolation function). We will thus rely on an extension of hyponormality to multiple operators. Following the definition of Athavale \cite{athavale1988}, operators $T_1, \hdots, T_n \in \mathcal{B}(\mathcal{H})$ are \textit{jointly hyponormal} if
%In the univariate case, the algorithm proposed in this paper computes a certain shift operator. For the algorithm to work, the shift must be hyponormal. As we will see, in our applications of interest, this can always be guaranteed. To tackle the multivariate case, we rely on an extension of the notion of hyponormality from a single operator to multiple operators.
\begin{equation}
\label{eq:hypo}
\begin{pmatrix}
[T_1^*,T_1] & [T_2^*,T_1] & \hdots & [T_n^*,T_1] \\
[T_1^*,T_2] & [T_2^*,T_2] & \hdots & [T_n^*,T_2] \\
\vdots & \vdots & & \vdots \\
[T_1^*,T_n] & [T_2^*,T_n] & \hdots & [T_n^*,T_n]
\end{pmatrix} \succcurlyeq 0
\end{equation}
in the sense that for all $u_1, \hdots u_n \in \mathcal{H}$, there holds
\begin{equation}
\sum\limits_{i,j=1}^n \langle u_i, [T_j^*,T_i] u_j \rangle \geqslant 0.
\end{equation}
This is equivalent to\footnote{This can seen with the Schur complement. If $A,B,C \in \mathcal{B}(\mathcal{H})$ and $A$ is invertible, then
$$ 
\begin{pmatrix}
A & B^* \\
B & C\hphantom{^*} 
\end{pmatrix}
=
\begin{pmatrix}
I & 0 \\
BA^{-1} & I 
\end{pmatrix}
\begin{pmatrix}
A & 0 \\
0 & C - BA^{-1}B^*
\end{pmatrix}
\begin{pmatrix}
I & A^{-1}B^* \\
0 & I
\end{pmatrix}.
$$}
\begin{equation}
\begin{pmatrix}
I & T_1^* & T_2^* & \hdots & T_n^*\\
T_1 & T_1^*T_1 & T_2^*T_1 & \hdots & T_n^* T_1 \\
T_2 & T_1^*T_2 & T_2^*T_2 & \hdots & T_n^* T_2 \\
\vdots & \vdots & \vdots & & \vdots \\
T_1 & T_1^*T_n & T_2^*T_n & \hdots & T_n^* T_n
\end{pmatrix} \succcurlyeq 0
\end{equation}
%
%$$ t^2 \langle u , A u \rangle + 2 t \Re \langle u , B^* v \rangle + \langle v , C v \rangle \geqslant 0 $$
%$$ | \Re \langle u , B^* v \rangle |^2 \leqslant \langle u , A u \rangle \langle v , C v \rangle $$
%$$ \langle v , C v \rangle - \langle v , B A^{-1} B^* v \rangle \geqslant 0 $$
%If $A = I$, then
%$$ \langle v , C v \rangle - \langle B^* v , B^* v \rangle \geqslant 0 $$
%If not,
%$$ \langle v , C v \rangle - \langle B^* v , A^{-1} B^* v \rangle \geqslant 0 $$
%
%$$ \sum_{i,j=1}^n \langle u_i , T_j^* T_i u_j \rangle + 2 \sum_{i=1}^n \Re \langle t u_0 , T_j^* u_j \rangle  + \|tu_0\|^2 \geqslant 0$$
%$$ \sum_{i,j=1}^n \langle T_j u_i , T_i u_j \rangle + 2t \sum_{i=1}^n \Re \langle u_0 , T_j^* u_j \rangle + t^2 \|u_0\|^2 \geqslant 0 $$
but it is stronger than requiring $t_1 T_1 + \hdots + t_n T_n $ to be hyponormal for all $t_1,\hdots,t_n \in \mathbb{C}$, that is
\begin{equation}
\sum_{i,j} ~ \overline{t_i} t_j ~ T_i^* T_j ~ \succcurlyeq ~ 0.
\end{equation}
The key ingredient to making the notion of joint hyponormality relevant for practical purposes is that in finite dimensions, the trace of $[T_i^*,T_i]$ is equal to zero, and thus
\begin{equation}
[T_i^*,T_i] \succcurlyeq 0 ~~~  \Longleftrightarrow ~~~  [T_i^*,T_i] = 0.
\end{equation}
This brings about the following equivalences in finite dimension which will be used throughout the paper:
%In finite dimensional setting, the following statements are thus equivalent:
\begin{enumerate}
\item $T_1,\hdots,T_n$ are jointly hyponormal.
\item The inequality in \eqref{eq:hypo} is an equality. 
\item $T_1,\hdots,T_n,T_1^*,\hdots,T_n^*$ commute pair-wise.
\item For all pairs $(T_i,T_j)$ with $i<j$, it holds that
\begin{equation}
\begin{pmatrix}
I & T_i^* & T_j^* \\
T_i & T_i^*T_i & T_i^*T_j  \\
T_j & T_i^*T_j & T_j^*T_j
\end{pmatrix} \succcurlyeq 0.
\end{equation}
\item There exists a unitary matrix $U$ (i.e. $U^*U=UU^*=I$ where $I$ denotes the identity) and diagonal matrices $D_1, \hdots, D_n$ such that 
\begin{equation}
\left\{
\begin{array}{c}
T_1 = UD_1U^*\\
T_2 = UD_2U^*\\
\vdots \\
T_n = UD_nU^*
\end{array}
\right.
\end{equation}
\end{enumerate}
These equivalences bring a new perspective on the moment problem to the best of our knowledge.  They were implicitely used in our preprint \cite{josz-molzahn-2015}. It allows one to deal with a kind of truncated moment problem that arises in practice, but which has not been given much attention to from a theoretical perspective. 

\section{Truncated moment problem}
\label{sec:Truncated moment problem}

The moment problem is an old yet active subject of research \cite{haviland1936,akhiezer-1965,landau1987,curto-2000,curtoquad-2000,curto-2010,putinar-1988,stochel-2001,curto-1996,curto-2005,kimsey2013,kimsey2016,stochel1998,putinar1992,curto1993,mourrain2009,helton2012,bucero2016,schmudgen2016,laurent2005,agler2002}. 
We now focus on a kind of moment problem that arises in practice. Consider an integer $d \in \mathbb{N}$ and an infinite sequence of complex numbers $(y_{\alpha,\beta})_{\alpha,\beta \in \mathbb{N}^n}$. In the context of polynomial optimization, we are interested in knowing whether there exists a positive Borel measure $\mu$ supported on the semi-algebraic set $K$ such that 
\begin{equation}
 y_{\alpha,\beta} = \int_{\mathbb{C}^n} \bar{z}^\alpha z^\beta d\mu ~, ~~~~ \text{for all} ~~ |\alpha|,|\beta| \leqslant d.
\end{equation}
Should it exist, we are also interested in computing the measure from its moments $(y_{\alpha,\beta})_{|\alpha|,|\beta|\leqslant d}$. In the context of exponential interpolation, a complex-valued measure supported on $\mathbb{C}^n$ is presumed to exist such that
\begin{equation}
 y_{\alpha,\beta} = \int_{\mathbb{C}^n} z^{\alpha+\beta} d\mu ~, ~~~~ \text{for all} ~~ |\alpha|,|\beta| \leqslant d
\end{equation}
and we are solely interested in computing the measure from its moments.

In previous work \cite{curto-2000,curto-1996,curto-2005}, the authors raise the question of whether there exists a positive Borel measure $\mu$ supported on the semi-algebraic set $K$ such that 
\begin{equation}
 y_{\alpha,\beta} = \int_{\mathbb{C}^n} \bar{z}^\alpha z^\beta d\mu ~, ~~~~ \text{for all} ~~ |\alpha|+|\beta| \leqslant 2d.
\end{equation}
This corresponds to a degree truncation, as opposed to a square truncation as above. The discrepancy is illustrated in Figure \ref{fig:trunc} below. It calls for different notions of moment matrices, as shown in Figure \ref{fig:moment}. The moment matrix resulting from square truncation is referred to as \textit{pruned} complex moment matrix in \cite{lasserre2007}, but the moment problem with square truncation is not considered. The moment problem with degree truncation is equivalent in some sense (see \cite[Theorem 5.2]{curto-2005}) to the even-dimensional real moment problem, i.e. where we seek a measure on a real semi-algebraic set such that
\begin{equation}
 y_{\alpha} = \int_{\mathbb{R}^{2n}} x^\alpha d\mu ~, ~~~~ \text{for all} ~~ |\alpha| \leqslant 2d
\end{equation}
given a real sequence $(y_{\alpha})_{\alpha \in \mathbb{N}^{2n}}$. In contrast, the square truncation we consider captures the real truncated moment problem as the special case (both even- and odd-dimensional). It corresponds to the case where the moment data forms a Hankel matrix (see \cite[Corollary 3.9]{josz-molzahn-2015} and Theorem \ref{th:toeplitz} below).

\begin{figure}[!h]
	\centering
	\includegraphics[width=.38\textwidth]{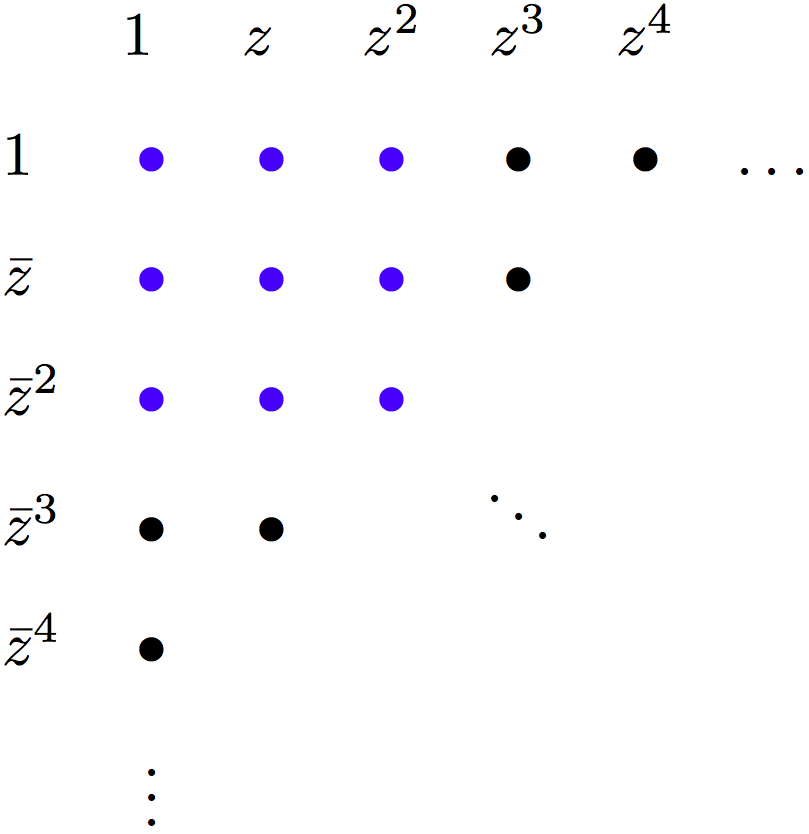}
	\caption{Square truncation in blue vs. degree truncation in black and blue}
	 \label{fig:trunc}
\end{figure}

\begin{figure}[!h]
	\centering
	\includegraphics[width=.75\textwidth]{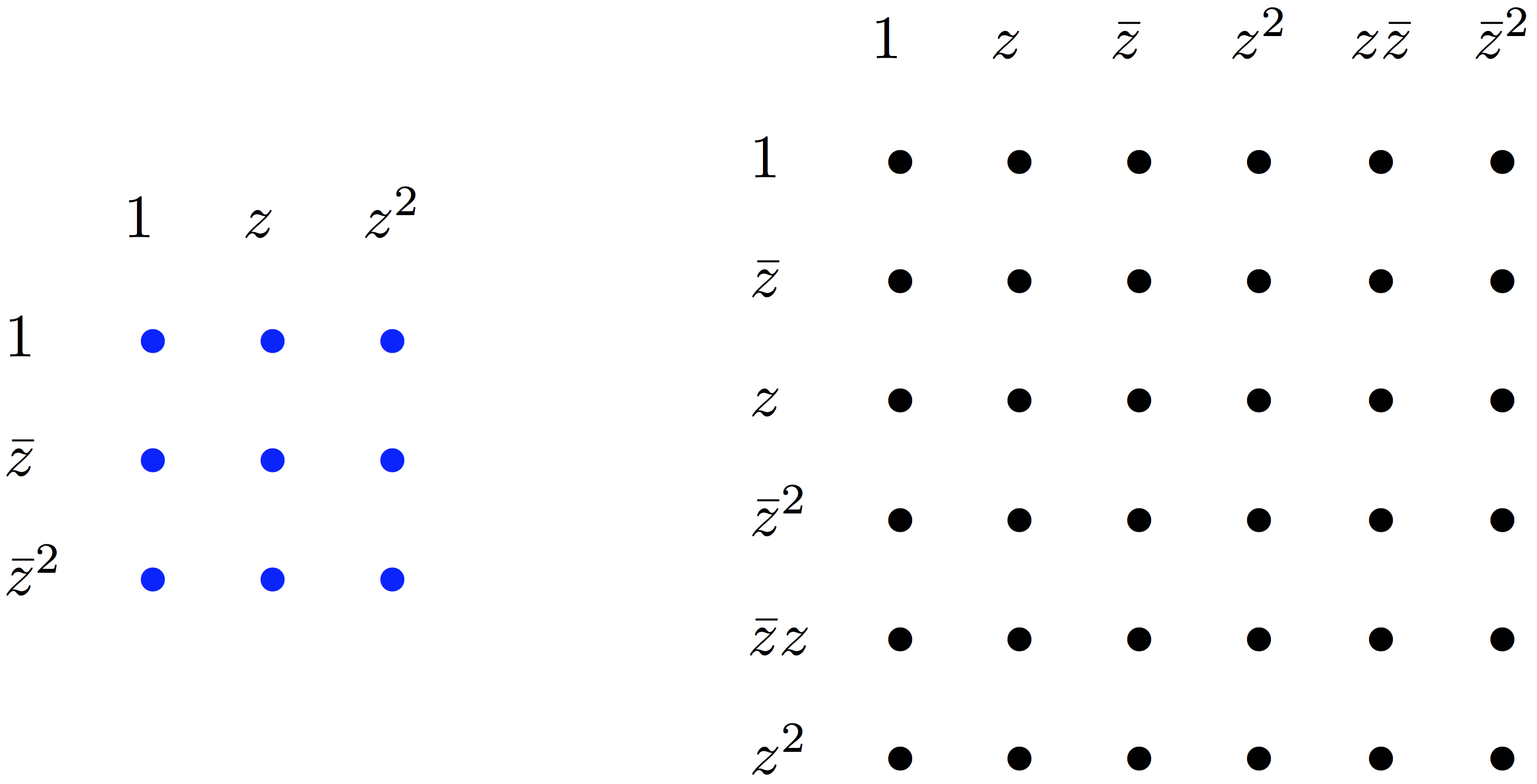}
	\caption{Moment matrix in the case of square truncation vs. moment matrix in the case of degree truncation}
	 \label{fig:moment}
\end{figure}

In Theorem \ref{th:trunc} and Theorem \ref{th:toeplitz} below, solutions to the moment problem with square truncation are given using the notion of hyponormality presented in Section \ref{sec:Joint hyponormality}. The proofs are constructive and the algorithm proposed in Section \ref{sec:Algorithm} replicates each step of the proofs. To introduce the results, we need some notation. In the multivariate case, we say that truncated data is Hermitian if $\overline{y_{\beta,\alpha}} = y_{\alpha,\beta}$ for all $|\alpha|,|\beta| \leqslant d$. Given an integer $r\in \mathbb{N}$, a $r$-atomic measure is sum of $r$ Dirac measures in $r$ distinct points (called \textit{atoms}) with nonzero weights. It is said to be positive if all the weights are positive, and supported on a set if the atoms lie in the set. In the setting of interpolation which we treat after Theorem \ref{th:trunc} and Theorem \ref{th:toeplitz}, the weights are complex numbers.

\begin{theorem}{\cite[Theorem 3.8]{josz-molzahn-2015}}
\label{th:trunc}
\normalfont
\textit{
Consider some complex numbers $(y_{\alpha,\beta})_{|\alpha|,|\beta| \leqslant d}$ with $d \geqslant d_K$ forming a Hermitian matrix. Assume that $K$ contains the constraint $\sum_{k=1}^n |z_k|^2 \leqslant R^2$ for some radius $R \geqslant 0$}.
\textit{Then there exists a positive} $ \text{rank} M_{d-d_K}(y)$\textit{-atomic measure} $\mu$ \textit{supported on} $K$ \textit{such that}
\begin{equation}
\label{eq:atomic_rep}
 y_{\alpha,\beta} = \int_{\mathbb{C}^n} \bar{z}^\alpha z^\beta d\mu ~, ~~~~ \text{for all} ~~ |\alpha|,|\beta| \leqslant d
\end{equation}
\textit{if and only if}
\begin{enumerate}
\item $M_d(y) \succcurlyeq 0$ and $M_{d-k_i}(g_iy) \succcurlyeq 0, ~ i = 1 \hdots m$;
\item $\text{rank} M_{d}(y) = \text{rank} M_{d-d_K}(y)$;
\item $ 
\begin{pmatrix} 
M_{d-d_K}(y) & M_{d-d_K}(\bar{z}_i y) & M_{d-d_K}(\bar{z}_j y) \\ 
M_{d-d_K}(z_i y) & M_{d-d_K}(|z_i|^2 y)  & M_{d-d_K}(\bar{z}_j z_i y) \\
M_{d-d_K}(z_j y) & M_{d-d_K}(\bar{z}_i z_j y) & M_{d-d_K}(|z_j|^2 y)
\end{pmatrix} 
\succcurlyeq 0 , ~ \forall 1 \leqslant i < j \leqslant n$.
\end{enumerate}
\textit{Moreover, $\mu$ is the unique} $ \text{rank} M_{d-d_K}(y)$\textit{-atomic measure} \textit{satisfying} \eqref{eq:atomic_rep}, \textit{and for each $1\leqslant i \leqslant m$, it has exactly} $\text{rank} M_{d}(y) - \text{rank} M_{d-d_K}(g_i y)$ \textit{atoms that are zeros of $g_i$. In the univariate case $n=1$,} Condition 3 \textit{must be replaced by}
\begin{equation}
\label{eq:uni}
\begin{pmatrix} 
M_{d-d_K}(y) & M_{d-d_K}(\bar{z} y) \\ 
M_{d-d_K}(z y) & M_{d-d_K}(|z|^2 y)
\end{pmatrix}
\succcurlyeq 0.
\end{equation}
\end{theorem}
\begin{proof}
We provide a sketch of the proof. We focus on the ``if" part, as it is important in the sequel. The positive semidefinite moment matrix of rank $r:=\text{rank} M_{d}(y)$ can be factorized as $M_d(y) = X^* X$ where $(\cdot)^*$ stands for adjoint, i.e. conjugate transpose. This can be achieved with the Cholesky factorization \cite{cholesky1910}. The rows of $X =: (x_{i,\alpha})$ are indexed by $1 \leqslant i \leqslant r$ and the columns are $X$ are indexed by $|\alpha| \leqslant d$. Thanks to Conditions 1 and 2, there exists complex matrices $T_1,\hdots,T_n$ of order $r$, called \textit{shift operators}, such that for each $1\leqslant k \leqslant n$, we have 
\begin{equation}
T_k x_\alpha = x_{\alpha+e_k}~,~~\forall |\alpha|\leqslant d-1.
\end{equation}
Here, $x_\alpha$ denotes the $\alpha$-column of $X$ and $e_k$ is the row vector of size $n$ that contains only zeros apart from 1 in position $k$. We now explain why the shift operators exist. Consider the finite-dimensional Hilbert space $\mathcal{H} := \text{span} ( x_{\alpha} )_{|\alpha|\leqslant d}$. Condition 2 implies that $\mathcal{H} = \text{span} ( x_{\alpha} )_{|\alpha|\leqslant d-1}$. Given some complex numbers $(u_{\alpha})_{|\alpha|\leqslant d-1}$, Condition 1 implies that
\begin{equation}
\left\| \sum\limits_{|\alpha|\leqslant d-1} u_{\alpha} x_{\alpha+e_k} \right\| ~~ \leqslant ~~  R ~ \left\| \sum\limits_{|\alpha|\leqslant d-1} u_{\alpha} x_{\alpha} \right\|
\end{equation}
where $\| x \| := \sqrt{x^*x}$ denote the 2-norm of a vector $x \in \mathbb{C}^r$.
As a result, given two possibly different sets of coefficients $(u_{\alpha})_{|\alpha|\leqslant d-1}$ and $(v_{\alpha})_{|\alpha|\leqslant d-1}$, if $\sum_{|\alpha| \leqslant d-1} u_\alpha x_{\alpha} = \sum_{|\alpha| \leqslant d-1} v_\alpha x_{\alpha}$, then $\sum_{|\alpha| \leqslant d-1} u_\alpha x_{\alpha+e_k} = \sum_{|\alpha| \leqslant d-1} v_\alpha x_{\alpha+e_k}$. In other words, each element of $\mathcal{H}$ has a unique image by $T_k$, which makes the shift well-defined. In fact, it is bounded by the radius $R$. 

Condition 3 implies that for all $1 \leqslant i < j \leqslant n$, we have
\begin{equation}
\begin{pmatrix}
I & T_i^* & T_j^* \\
T_i & T_i^*T_i & T_i^*T_j  \\
T_j & T_i^*T_j & T_j^*T_j
\end{pmatrix} \succcurlyeq 0.
\end{equation}
As discussed in Section \ref{sec:Joint hyponormality}, this makes the shifts jointly hyponormal. Thus there exists a unitary matrix $P$ such that $T_k = PD_kP^*$ where $D_k =: \text{diag}(d_{k1},\hdots,d_{kr})$ is a diagonal matrix for each $1\leqslant k \leqslant n$. Bearing in mind that $M_d(y) = X^*X$, we have for all $|\alpha|,|\beta|\leqslant d$ that
\begin{equation}
\begin{array}{rcl}
y_{\alpha,\beta} & = & x_\alpha^* x_\beta \\[.5em]
& = & (T^\alpha x_0)^* (T^\beta x_0 )\\[.5em]
& = & x_0^* (T^\alpha)^* T^\beta x_0 \\[.5em]
& = & x_0^* (PD^\alpha P^*)^* PD^\beta P^* x_0 \\[.5em]
& = & x_0^* P \overline{D}^\alpha P^*P D^\beta P^* x_0 \\[.5em]
& = & x_0^* P \overline{D}^\alpha D^\beta P^* x_0 \\[.5em]
& = & x_0^* \left( \sum\limits_{k=1}^r p_k \overline{d}_k^\alpha d_k^\beta p_k^* \right) x_0 \\[1em]
& = & \sum\limits_{k=1}^r x_0^* p_k p_k^* x_0 ~ \bar{d}_k^\alpha d_k^\beta \\[1em]
& = & \sum\limits_{k=1}^r |x_0^* p_k|^2 ~ \bar{d}_k^\alpha d_k^\beta
\end{array}
\end{equation}
where $P =: (p_1 \hdots p_r)$ denote the columns of $P$ and $d_j := (d_{1j},\hdots,d_{nj})$.
As a result, eigenvalues of the shift operators correspond to the support of a measure, and their eigenvectors yield the weights of a measure. Precisely, the following measure 
\begin{equation}
\mu = \sum_{k=1}^r  |x_0^* p_k|^2 ~ \delta_{d_k}
\end{equation}
solves the truncated problem up to order $d$, i.e. \eqref{eq:atomic_rep}. The uniqueness of the measure can easily be deduced from Lemma \ref{lemma:ind} and Lemma \ref{lemma:range} in the Appendix.
\end{proof}

We will say that the truncated moment data $(y_{\alpha,\beta})_{|\alpha|,|\beta| \leqslant d}$ is \textit{hyponormal} if it satisfies Condition 3 of Theorem \ref{th:trunc}. Indeed, this condition corresponds to the joint hyponormality of the shift operators. We propose to enforce the joint hyponormality of the shift operators in the complex hierarchy by requiring the truncated data to be hypornormal. This yields the following primal problem
\begin{equation}
\begin{array}{l}
\inf_y        L_y(f) ~~~ \text{s.t.}~~~ M_d(y) \succcurlyeq 0 ~~,~~ M_{d-k_j}(g_iy) \succcurlyeq 0, ~i = 1, \hdots, m, ~~~\text{and} \\[1em]
\begin{pmatrix} 
M_{d-d_K}(y) & M_{d-d_K}(z_i y) & M_{d-d_K}(z_j y) \\ 
M_{d-d_K}(\bar{z}_i y) & M_{d-d_K}(|z_i|^2 y)  & M_{d-d_K}(z_j \bar{z}_i y) \\
M_{d-d_K}(\bar{z}_j y) & M_{d-d_K}( z_i \bar{z}_j y) & M_{d-d_K}(|z_j|^2 y)
\end{pmatrix} 
\succcurlyeq 0 , ~ \forall 1 \leqslant i < j \leqslant n
\end{array}
\end{equation}
and its dual counterpart
\begin{equation}
\begin{array}{c}
\sup_{\lambda \in \mathbb{R}, \sigma_j \in \Sigma_d[z]} \lambda ~~~ \text{s.t.} ~~~ f - \lambda = \sigma_0 + \sigma_1 g_1 + \hdots + \sigma_m g_m + \hdots \\[1em]
 \sum\limits_{1 \leqslant i < j \leqslant n} \sigma^{11}_{ij} + \sigma^{12}_{ij} \bar{z}_i + \sigma^{21}_{ij} z_i + \sigma^{13}_{ij} \bar{z}_j + \sigma^{31}_{ij} z_j + \hdots \\[1.2em]
\sigma^{22}_{ij} |z_i|^2 + \sigma^{23}_{ij} \bar{z}_j z_i + \sigma^{32}_{ij} \bar{z}_i z_j + \sigma^{33}_{ij}|z_j|^2 \\[1em]
\begin{pmatrix} 
\sigma^{11}_{ij} & \sigma^{12}_{ij} & \sigma^{13}_{ij} \\[.5em] 
\sigma^{21}_{ij} & \sigma^{22}_{ij}  & \sigma^{23}_{ij} \\[.5em]
\sigma^{31}_{ij} & \sigma^{32}_{ij} & \sigma^{33}_{ij}
\end{pmatrix} 
\succcurlyeq 0 , ~ \forall 1 \leqslant i < j \leqslant n,\\[2.5em]
\sigma^{11}_{ij}, \sigma^{12}_{ij}, \sigma^{13}_{ij},
\sigma^{21}_{ij}, \sigma^{22}_{ij} , \sigma^{23}_{ij},
\sigma^{31}_{ij}, \sigma^{32}_{ij}, \sigma^{33}_{ij} \in \mathbb{R}_{d-d_k}[z,\bar{z}]
\end{array}
\end{equation}
where a polynomial $p$ belongs to $\mathbb{R}_{d}[z,\bar{z}]$ if it is of the form $\sum_{|\alpha|,|\beta|\leqslant d} p_{\alpha,\beta} \bar{z}^\alpha z^\beta$ where $p_{\alpha,\beta} \in \mathbb{C}$. Recall that for complex polynomials, the coefficients $p_{\alpha,\beta}$ are unique (see Section 3.2, footnote 4 in \cite{josz-molzahn-2015} for a comparison with real polynomials). Thus the complex polynomial $p$ can be identified with the Hermitian matrix $(p_{\alpha,\beta})_{|\alpha|,|\beta|\leqslant d}$. This is what we do in the semidefinite constraints in the above dual problem.
In Example \ref{eg:ellipsebis}, we apply this hierarchy of relaxations to a bivariate optimization problem. Compared to the complex moment/sum-of-squares hierarchy, it solves this particular problem at a lower order.

There exists several identifiable cases where the shift operators are naturally jointly hyponormal. It was noticed in \cite[Corollary 3.9]{josz-molzahn-2015} that it is the case when the truncated data forms a Hermitian Hankel matrix (thus real-valued). This corresponds exactly to the truncated data generated by the Lasserre hierarchy for real polynomial optimization. In that case, the shift operators are real symmetric. Below, we show that hyponormality is also guaranteed if we assume that the truncated data forms a Toeplitz matrix. This result has not been presented in the literature to the best of our knowledge. In the multivariate case, we say that the truncated data is Hankel if $y_{\alpha,\beta} = y_{\gamma,\delta}$ for all $|\alpha|,|\beta|,|\gamma|,|\delta| \leqslant d$ such that $\alpha+\beta=\gamma+\delta$, and we say that the truncated data is Toeplitz if $y_{\alpha,\beta} = y_{\gamma,\delta}$ for all $|\alpha|,|\beta|,|\gamma|,|\delta| \leqslant d$ such that $\alpha-\beta=\gamma-\delta$. In other words, $y_{\alpha,\beta}$ only depends on $\alpha+\beta$ in a Hankel matrix, and it only depends on $\alpha-\beta$ in a Toeplitz matrix.
\begin{theorem}
\label{th:toeplitz}
\normalfont
\textit{
Consider some complex numbers $(y_{\alpha,\beta})_{|\alpha|,|\beta| \leqslant d}$ with $d\geqslant d_K$ forming either a Hermitian Toeplitz matrix or a Hermitian Hankel matrix.}
\textit{There exists a positive} $ \text{rank} M_{d-d_K}(y)$\textit{-atomic measure} $\mu$ \textit{supported on} $K$ \textit{such that}
\begin{equation}
\label{eq:atomic_repbis}
 y_{\alpha,\beta} = \int_{\mathbb{C}^n} \bar{z}^\alpha z^\beta d\mu ~, ~~~~ \text{for all} ~~ |\alpha|,|\beta| \leqslant d
\end{equation}
\textit{if and only if}
\begin{enumerate}
\item $M_d(y) \succcurlyeq 0$ and $M_{d-d_K}(g_iy) \succcurlyeq 0, ~ i = 1 \hdots m$;
\item $\text{rank} M_{d}(y) = \text{rank} M_{d-d_K}(y)$.
\end{enumerate}
\textit{Moreover, $\mu$ is the unique} $ \text{rank} M_{d-d_K}(y)$\textit{-atomic measure} \textit{satisfying \eqref{eq:atomic_repbis}, and for each $1\leqslant i \leqslant m$, the measure $\mu$ has exactly} $\text{rank} M_{d}(y) - \text{rank} M_{d-d_K}(g_i y)$ \textit{atoms that are zeros of $g_i$.}
\end{theorem}
\begin{proof}
($\Longrightarrow$) Same as proof as in \cite[Theorem 3.8]{josz-molzahn-2015}. ($\Longleftarrow$) The proof is similar to that of Theorem \ref{th:trunc}. We focus on the areas where it differs and only consider the case of Toeplitz matrices. 
The Toeplitz property implies that the shift operators in the proof of Theorem~\ref{th:trunc} are well-defined and unitary. Indeed, for all complex numbers $(u_\alpha)_{|\alpha|\leqslant d-1}$, it holds that
\begin{equation}
\begin{array}{rcl}
\left\| \sum\limits_{|\alpha| \leqslant d-1} u_\alpha x_{\alpha+e_k} \right\|^2 & ~=~ & \sum\limits_{|\alpha|,|\beta|\leqslant d-1} \overline{u}_\alpha u_\beta ~ x_{\alpha+e_k}^* x_{\beta+e_k} \\[2em]
& ~=~ & \sum\limits_{|\alpha|,|\beta|\leqslant d-1} \overline{u}_\alpha u_\beta ~ y_{\alpha+e_k,\beta+e_k} \\[2em]
& ~=~ & \sum\limits_{|\alpha|,|\beta|\leqslant d-1} \overline{u}_\alpha u_\beta ~ y_{\alpha,\beta} \\[2em]
& ~=~ &  \left\| \sum\limits_{|\alpha| \leqslant d-1} u_\alpha x_{\alpha} \right\|^2.
\end{array}
\end{equation}
Using the same argument as in the proof of Theorem \ref{th:trunc}, the shift $T_k$ is well-defined. In addition, it is isometric and thus satisfies $T_k^* T_k = T_k T_k^* = I$.
Now, observe that $(T_1,\hdots,T_n)$ is a pair-wise commuting tuple of operators on the Hilbert space $\mathcal{H}$. As a consequence, $(T_1, \hdots , T_n,T_1^*, \hdots, T_n^*) = ( T_1, \hdots , T_n,T_1^{-1}, \hdots, T_n^{-1})$ is also a pair-wise commuting tuple of operators. Indeed, if two invertible square matrices $A$ and $B$ commute, so do $A^{-1}$ and $B^{-1}$ (since $A^{-1}B^{-1} AB B^{-1}A^{-1} =  A^{-1}B^{-1} BA B^{-1}A^{-1}$), and so do $A$ and $B^{-1}$ (since $B^{-1} AB B^{-1} =  B^{-1} BA B^{-1}$). It follows that $T_1,\hdots,T_n$ are jointly hyponormal. The rest of the proof is identical to that of Theorem \ref{th:trunc}.
\end{proof}

In the univariate case $n=1$ with support equal to the full space $K = \mathbb{C}$, the truncated moment problem in Theorem \ref{th:toeplitz} with Toeplitz data is actually the truncated trigonometric moment problem. A solution to this problem has been given by \cite[P. 211]{iohvidov1982}, \cite[Theorem I.I.12]{akhiezer1962}, and \cite[Theorem 6.12]{curto1991}. It can be stated as follows. A Toeplitz matrix with positive upper left element is positive semidefinite if and only if it is represented by a positive Borel measure. In other words, the rank need not be preserved (Condition 2 of Theorem \ref{th:toeplitz}) for there to exist a measure. The trigonometric moment problem has been considered more recently in \cite{gabardo1999,li2006,zag2015}. \\

In the setting of interpolation, the shift operators are more simple to study since we assume the existence of a representing measure (i.e. $\nu = \sum_{k=1}^d w_k \delta_{\exp(f_k)}$). The measure is uniquely determined by its moments $y_\alpha$ up to degree $2d$, and thus the interpolation problem has a unique solution with $d$ exponentials. In addition, the rank is preserved, i.e. $\text{rank} \mathcal{H}_{d}(y) = \text{rank} \mathcal{H}_{d-1}(y)$. These claims can be easily be deduced from Lemma \ref{lemma:ind} and Lemma \ref{lemma:range} in the Appendix. We next prove that shift operators associated to the Hankel moment matrix are guaranteed to exist, that they are simultaneously diagonalizable, and that they are complex symmetric.

The moment matrix $\mathcal{H}_d(y)$ is of Hankel type and is thus complex symmetric. According to the Autonne-Takagi factorization \cite{autonne} \cite[Theorem II]{takagi} which applies to any square complex symmetric matrix, there exists a unitary matrix $U$ (i.e. $U^*U = UU^* = I$) and a diagonal matrix $D$ with real nonnegative entries such that $\mathcal{H}_d(y) = UDU^T$. Note that the diagonal values of $D$ are the eigenvalues of $\mathcal{H}_d(y)\mathcal{H}_d(y)^*$, whose rank is equal to that of $\mathcal{H}_d(y)$. Defining $X := \sqrt{D}~U^T$, we may in fact write that $\mathcal{H}_d(y) = X^T X$ where the rows of $X =: (x_{i,\alpha})$ are indexed by $1 \leqslant i \leqslant d$ and the columns are $X$ are indexed by $|\alpha| \leqslant d$. In addition, since our data is represented by the measure $\nu$, we know that $\mathcal{H}_d(y) = V^TV$ where the rows of $V =: (s_k \exp(f_k)^\alpha)$ are indexed by $1 \leqslant k \leqslant d$ and the columns are $V$ are indexed by $|\alpha| \leqslant d$. Here, $s_k$ is a complex number such that $s_k^2 = w_k$ where $w_k$ is a complex weight of the measure $\nu$. As a result, $X^T X = V^T V$ and thus $X$ and $V$ have the same ranges. Hence, there exists an invertible matrix $P$ such that $X = PV$. Thus $V^T P^T P V = V^T V$, or, in other words, $V^T (P^TP - I)V = 0$. According to Lemma \ref{lemma:ind} and Lemma \ref{lemma:range} in the Appendix, the range of $V$ is equal to $\mathbb{C}^d$, so that we in fact have $v^T (P^TP - I)v = 0$ for all $v \in \mathbb{C}^d$. As a result, $P^T P = P P^T = I$. Going back to the relationship $X = PV$, we have in particular that $x_\alpha = P v_\alpha$ for all $|\alpha| \leqslant d$. Defining $D_k := \text{diag}( \exp(f_{1k}) , \hdots , \exp(f_{dk}))$, we have that $D_k v_\alpha = v_{\alpha+e_k}$, and thus $PD_kP^Tx_\alpha = x_{\alpha+e_k}$. The shift operators are thus $T_k := PD_kP^T$ for $1\leqslant k \leqslant n$. Remembering that $\mathcal{H}_d(y) = X^T X$, we have
\begin{equation}
\begin{array}{rcl}
y_{\alpha,\beta} & = & x_\alpha^T x_\beta \\[.5em]
& = & (T^\alpha x_0)^T (T^\beta x_0 )\\[.5em]
& = & x_0^T (T^\alpha)^T T^\beta x_0 \\[.5em]
& = & x_0^T (PD^\alpha P^T)^T PD^\beta P^T x_0 \\[.5em]
& = & x_0^T P D^\alpha P^TP D^\beta P^T x_0 \\[.5em]
& = & x_0^T P D^\alpha  D^\beta P^T x_0 \\[.5em]
& = & x_0^T P D^{\alpha+\beta} P^T x_0 \\[.25em]
& = &  x_0^T\left( \sum\limits_{k=1}^d p_k d_k^{\alpha+\beta} p_k^T\right) x_0 \\[.5em]
& = &  \sum\limits_{k=1}^d (x_0^T p_k p_k^T x_0) ~ d_k^{\alpha+\beta}  \\[.5em]
& = &  \sum\limits_{k=1}^d (x_0^T p_k)^2 ~ d_k^{\alpha+\beta}
\end{array}
\end{equation}
so that the sought measure is entirely determined
\begin{equation}
\nu = \sum_{k=1}^d (x_0^T p_k)^2 ~ \delta_{d_k}.
\end{equation}
We conclude by proving that the shifts are complex symmetric. Since $\text{rank} \mathcal{H}_{d}(y) = \text{rank} \mathcal{H}_{d-1}(y)$, we have that $\text{span} (x_\alpha)_{|\alpha|\leqslant d} = \text{span} (x_\alpha)_{|\alpha|\leqslant d-1}$. Consider some complex numbers $(u_\alpha)_{|\alpha|\leqslant d-1}$ and $(v_\alpha)_{|\alpha|\leqslant d-1}$. Let $u := \sum_{|\alpha|\leqslant d-1} u_\alpha x_\alpha$ and $v:= \sum_{|\alpha|\leqslant d-1} v_\alpha x_\alpha$ and compute
\begin{equation}
\begin{array}{rcl}
u^T T_k v & = & (\sum_{|\alpha|\leqslant d-1} u_\alpha x_\alpha)^T T_k (\sum_{|\alpha|\leqslant d-1} v_\alpha x_\alpha) \\[.5em]
& = & (\sum_{|\alpha|\leqslant d-1} u_\alpha x_\alpha)^T (\sum_{|\alpha|\leqslant d-1} v_\alpha T_kx_\alpha) \\[.5em]
& = & (\sum_{|\alpha|\leqslant d-1} u_\alpha x_\alpha)^T (\sum_{|\alpha|\leqslant d-1} v_\alpha x_{\alpha+e_k}) \\[.5em]
& = & \sum_{|\alpha|,|\beta|\leqslant d-1} u_\alpha v_\beta ~ x_\alpha^T x_{\beta+e_k} \\[.5em]
& = & \sum_{|\alpha|,|\beta|\leqslant d-1} u_\alpha v_\beta ~ y_{\alpha+\beta+e_k} \\[.5em]
& = & \sum_{|\alpha|,|\beta|\leqslant d-1} u_\alpha v_\beta ~ x_{\alpha+e_k}^T x_{\beta} \\[.5em]
 & = & (T_k u)^T v = u^T T_k^T v 
\end{array}
\end{equation}
whence $T_k = T_k^T$.

\section{Algorithm}
\label{sec:Algorithm}
\text{}
\\
Input: 
\begin{itemize}
\item number of variables $n$
\item truncation order $d$
\item moment matrix of rank $r$
\end{itemize}
Output:
\begin{itemize}
\item $r$-atomic measure\\
\end{itemize}
Below, the notation $\bullet$ either stands either for conjugate transpose or transpose depending on whether the algorithm is being applied to polynomial optimization ($\bullet = *$) or interpolation ($\bullet = T$). In the case of real polynomial optimization, $\bullet = * = T$.\\
\begin{enumerate}
\item Factorize the moment matrix into $X^{\bullet}X$. The rows of $X = (x_{i,\alpha})$ are indexed by $1\leqslant i \leqslant r$ and the columns of $X$ are indexed by $|\alpha|\leqslant d$.\\
\item Find a subset of the columns of $X$ that generate the column space of $X$. Let $\alpha(1),\hdots,\alpha(r) \in \mathbb{N}^n$ denote their indexes.\\
\item Compute the shift operators $T_1, \hdots , T_n \in \mathbb{C}^{r \times r}$ by applying them only to the column basis, i.e. $T_k x_{\alpha(i)} = x_{\alpha(i)+e_k}$ for $1\leqslant i \leqslant r$.\\
%in the canonical basis $b_1,\hdots,b_p$ of $\mathbb{C}^p$; to do so, find $(c_{i,j})_{1\leqslant i,j \leqslant p}$ such that $b_i = \sum_{j=1}^p c_{ij} ( x_{\alpha(j),1}, \hdots, x_{\alpha(j),p} )$ and compute $T_kb_i = \sum_{j=1}^p c_{i,j} ( x_{\alpha(j)+e_k,1}, \hdots, x_{\alpha(j)+e_k,p} )$.\\
\item Generate some random $t_1, \hdots , t_n \in \mathbb{R}$ and diagonalize the matrix $\sum_{i=1}^n t_i T_i = PDP^\bullet$ with $P^\bullet P = P P^\bullet = I$. Let $P =: (p_1 \hdots p_r)$ denote the columns of $P$.\\
\item Compute the measure $\mu  =  \sum\limits_{k=1}^r x_0^{\bullet} p_k p_k^{\bullet} x_0 ~ \delta_{(p_k^\bullet T_i p_k)_{1\leqslant i \leqslant n}}$.
\end{enumerate}
In the case of polynomial optimization, the atoms of the measure are global solutions. In the case of interpolation, the arguments of the atoms are the frequencies (modulo $2\pi$ for their imaginary parts) and the weights of the measure are the weights of the complex exponential sum.

\section{Numerical experiments}
\label{sec:Numerical experiments}
We use Matlab 2015b, CVX \cite{cvx,gb08}, and SeDuMi \cite{sturm-1999}. The Cholesky factorization of positive semidefinite matrices is computed via an eigendecomposition followed by a QR factorization. The Autonne-Takagi factorization is computed using the implementation of Guo, Luk, Xu, and Piao \cite{xu2006,xu2005,guo2003,luk2003}. Another algorithm for this factorization is discussed in \cite{gerstner1988}. Table \ref{tab:123} below summarizes the experiments. Each of them illustrates a different property of the shift operators that arises in applications.
\begin{table}[h]
\begin{tabular}{|c|c|c|}
\hline
\textbf{Truncated data} & \textbf{Shift operators} & \textbf{Experiment}  \\
\hline
General case (only Hermitian) & Existence not guaranteed & Example \ref{eg:exist} \\
\hline
Hermitian but not hyponormal & $T_1,\hdots,T_n,T_1^*,\hdots,T_n^*$ do not commute & Example \ref{eg:commute} \\
\hline
Hyponormal & $T_1,\hdots,T_n,T_1^*,\hdots,T_n^*$ commute & Example \ref{eg:ellipse} \\
\hline
Enforced hyponormality & $T_1,\hdots,T_n,T_1^*,\hdots,T_n^*$ are made to commute & Example \ref{eg:ellipsebis} \\
\hline
Hermitian Toeplitz & Unitary $T_k^* = T_k^{-1}, ~ k = 1, \hdots, n$ & Example \ref{eg:toeplitz} \\
\hline 
Hermitian Hankel & Real symmetric $T_k^T = T_k, ~ k = 1, \hdots, n$ & Example \ref{eg:real} \\
\hline 
Complex Hankel & Complex symmetric $T_k^T = T_k, ~ k = 1, \hdots, n$ & Example \ref{eg:prony} \\
\hline 
\end{tabular}
\caption{Properties of shift operators}
\label{tab:123}
\end{table}

\begin{example}[Nonexistent shifts] 
\label{eg:exist}
Consider the following truncated data:
\begin{equation}
M_2(y) ~=~ 
\begin{array}{cccc}
 & ~1~ & ~z~ & ~z^2~ \\[.5em]
1\hphantom{^2}  & 1 & 1 & 2\hphantom{^2} \\[.5em]
\bar{z}\hphantom{^2} &  1 & 1 & 2\hphantom{^2} \\[.5em]
\bar{z}^2 & 2 & 2 & 4\hphantom{^2}
\end{array}
\end{equation}
The existence of shift operators is guaranteed if Conditions 1 and 2 of Theorem \ref{th:trunc} hold. Condition 2 is satified because $\text{rank}M_1(y) = \text{rank}M_2(y) = 1$. However, Condition 1 is not satisfied. The spectrum of $M_2(y)$ is equal to $\{ 0 , 6 \}$ but there does not exist $R\geqslant 0$ such that $M_1[(R^2-|z|^2)y] \succcurlyeq 0$. Indeed, the characteristic polynomial of $M_1[(R^2-|z|^2)y]$ with indeterminate $X$ is equal to
\begin{equation}
\left|
\begin{array}{cc}
X + 1 - R^2 & 2 - R^2 \\
2 - R^2 & X + 4 - R^2 
\end{array}
\right|
= \hdots 
\end{equation}
$$
\left( X - \frac{2R^2 - 5 + \sqrt{ 4R^2 - 20R^2 + 29}}{2} \right) \left( X - \underbrace{\frac{2R^2 - 5 - \sqrt{ 4R^2 - 20R^2 + 29}}{2}}_{< 0} \right).
$$
The moment matrix can be factorized as $M_2(y) = X^* X$ where
\begin{equation}
X ~=~ \begin{pmatrix} x_0 & x_1 & x_2 \end{pmatrix} =
\begin{array}{ccc}
~\hphantom{^2}1 & ~z~ & z^2~ \\
( \hphantom{^2}1 & 1 & 2\hphantom{^2} )
\end{array}
%\left(
%\begin{array}{c}
% x_0 \\[.5em]
%  x_1 \\[.5em]
%x_2
%\end{array}
%\right)
%:=
%\begin{array}{c}
%1\hphantom{^2}   \\[.5em]
%z\hphantom{^2}  \\[.5em]
%z^2 
%\end{array}
%\left(
%\begin{array}{c}
% 1 \\[.5em]
%  1 \\[.5em]
% 2
%\end{array}
%\right)
\end{equation}
There does not exist a shift operator $T$ acting on $\mathbb{C}$, i.e. a scalar $T\in \mathbb{C}$, such that $T x_0 = x_1$ and $Tx_1 = x_2$. That would imply that $T \times 1 = 1$ and $T \times 1 = 2 $, which is absurd.
%Condition 3 is not satisfied because
%\begin{equation}
%\text{sp}\left\{
%\begin{pmatrix} 
%M_{1}(y) & M_{1}(z y) \\ 
%M_{1}(\bar{z} y) & M_{1}(|z|^2 y)
%\end{pmatrix} 
%\right\}
%=
%\text{sp}\left\{
%\begin{pmatrix} 
%1 & 1 & 1 & 1 \\
%1 & 1 & 2 & 2 \\
%1 & 2 & 1 & 2 \\
%1 & 2 & 2 & 4
%\end{pmatrix} 
%\right\}
%=
%\begin{pmatrix}
%   -1.0000\\
%    \hphantom{-}0.1459\\
%    \hphantom{-}1.0000\\
%    \hphantom{-}6.8541
%\end{pmatrix}
%\end{equation} 
\end{example}

\begin{example}[Non-hyponormal shifts] 
\label{eg:commute}
Consider the following truncated data which was randomly generated:
$$
M_2(y) = \hdots
$$
$$
\scriptsize  
\begin{array}{ccccccc}
 & 1 & z_1 & z_2 & z_1^2 & z_1 z_2 & z_2^2 \\[.5em]
1 & \hphantom{-}5.6315 - 0.0000i  & \hphantom{-}1.5971 - 1.5041i &  -3.2028 + 0.0950i & -6.8376 - 2.1324i & \hphantom{-}7.6557 - 6.1320i & -1.2416 - 0.9424i  \\[.5em]
\bar{z}_1 & \hphantom{-}1.5971 + 1.5041i &  \hphantom{-}2.6137 - 0.0000i &  -2.0895 - 0.2555i &  -2.0629 - 2.5862i & \hphantom{-}5.3658 - 1.4552i & -1.7395 - 0.0898i \\[.5em]
\bar{z}_2 & -3.2028 - 0.0950i & -2.0895 + 0.2555i & \hphantom{-}4.2547 - 0.0000i &  \hphantom{-}4.9178 + 1.7360i & -7.3334 + 5.0916i  & \hphantom{-}2.2232 + 2.0478i \\[.5em]
\bar{z}_1^2 & -6.8376 + 2.1324i & -2.0629 + 2.5862i & \hphantom{-}4.9178 - 1.7360i  & \hphantom{-}9.6940 - 0.0000i & -7.9417 +11.7260i  & \hphantom{-}2.6639 + 0.8871i \\[.5em]
\bar{z}_1\bar{z}_2 & \hphantom{-}7.6557 + 6.1320i &  \hphantom{-}5.3658 + 1.4552i & -7.3334 - 5.0916i & -7.9417 -11.7260i & \hphantom{-}22.1171 - 0.0000i & -1.9300 - 5.1471i \\[.5em]
\bar{z}_2^2 & -1.2416 + 0.9424i & -1.7395 + 0.0898i & \hphantom{-}2.2232 - 2.0478i  & \hphantom{-}2.6639 - 0.8871i & -1.9300 + 5.1471i &  \hphantom{-}3.2855 - 0.0000i
\end{array}
$$
We chose the data so that $M_2(y) \succcurlyeq 0$ and $M_1(y)$ is invertible. As result, $M_1[(R^2-|z_1|^2-|z_2|^2)y] \succcurlyeq 0$ holds as long a $R>0$ is big enough. Hence Condition 1 of Theorem \ref{th:trunc} holds. In addition, we chose the data so that the rank is preserved, i.e. $\text{rank} M_1(y) = \text{rank} M_2(y) = 3$ and so Condition 2 holds. All the conditions of Theorem \ref{th:trunc} are thus satisfied, apart from Condition 3. We now show that Condition 3 fails to hold, and thus no atomic measure may be extracted from the data.

We have the Cholesky decomposition $M_2(y) = X^*X$ where 
$$X = \hdots $$
$$
\small
\begin{array}{cccccc}
1 & z_1 & z_2 & z_1^2 & z_1 z_2 & z_2^2 \\[.5em]
\hphantom{-}2.3731 - 0.0000i\hphantom{-} & \hphantom{-}0.6730 + 0.6338i\hphantom{-} & -1.3496 - 0.0400i\hphantom{-} & -2.8813 + 0.8986i\hphantom{-} & \hphantom{-}3.2261 + 2.5840i\hphantom{-} & -0.5232 + 0.3971i\hphantom{-} \\[.5em]
\hphantom{-}0.0000 - 0.0000i\hphantom{-} & \hphantom{-}1.3263 - 0.0000i\hphantom{-} & -0.8715 - 0.4321i\hphantom{-} & -0.5227 + 0.1170i\hphantom{-} & \hphantom{-}1.1738 + 1.3277i\hphantom{-} & -1.2358 - 0.3839i\hphantom{-} \\[.5em]
\hphantom{-}0.0000 - 0.0000i\hphantom{-} & \hphantom{-}0.0000 - 0.0000i\hphantom{-}  & \hphantom{-}1.2188 - 0.0000i\hphantom{-} & \hphantom{-}0.5416 - 0.0658i\hphantom{-} & -1.0497 - 0.8889i\hphantom{-} & \hphantom{-}0.2380 - 1.0596i\hphantom{-}
\end{array}
$$
for which the column basis indexed by $\{1,z_1,z_2\}$ is readily identified. We then obtain the shift operators
\begin{equation}
T_1 =
\left(
\begin{array}{ccc}
   0.2836 + 0.2671i & -2.1888 + 0.4064i  & \hphantom{-}1.2431 + 1.9399i \\[.5em]
   0.5589 - 0.0000i & -0.6777 - 0.1789i &  \hphantom{-}1.1608 + 0.7396i\\[.5em]
   0.0000 - 0.0000i  & \hphantom{-}0.4084 - 0.0496i & -0.5517 - 0.6200i\\[.5em]
\end{array}
\right)
\end{equation}
and
\begin{equation}
T_2 =
\left(
\begin{array}{ccc}
  -0.5687 - 0.0169i &  \hphantom{-}2.7130 + 2.2286i &  \hphantom{-}0.0913 + 2.8438i\\[.5em]
  -0.3672 - 0.1821i & \hphantom{-}0.9844 + 1.2690i & -1.1607 + 0.7277i\\[.5em]
   \hphantom{-}0.5136 - 0.0000i & -1.0521 - 0.9157i  & \hphantom{-}0.3364 - 1.8803i\\[.5em]
\end{array}
\right).
\end{equation}
For example, the image by $T_1$ of the column of $X$ indexed $z_2$ is equal to the column of $X$ indexed by $z_1z_2$:
\begin{equation}
T_1 \times
\begin{pmatrix}
-1.3496 - 0.0400i\\
-0.8715 - 0.4321i\\
\hphantom{-}1.2188 - 0.0000i
\end{pmatrix}
=
\begin{pmatrix}
   \hphantom{-}3.2261 + 2.5840i \\
   \hphantom{-}1.1738 + 1.3277i \\
  -1.0497 - 0.8889i 
\end{pmatrix}.
\end{equation}
Using an eigendecomposition, we find that
\begin{equation}
\text{sp}\left\{
\begin{pmatrix}
I        & T_1^* & T_2^* \\
T_1    & T_1^* T_1  & T_2^* T_1  \\
T_2    & T_1^* T_2  & T_2^* T_2
\end{pmatrix} 
\right\}
= 
\begin{pmatrix}
-18.4798\\
-4.4504\\ 
-2.9400\\
\hphantom{-}0.9867\\
\hphantom{-}3.9620\\ 
\hphantom{-} 5.4116\\ 
\hphantom{-}13.3779\\ 
\hphantom{-}20.0161\\
\hphantom{-}30.3167
\end{pmatrix}
\end{equation}
and
\begin{equation}
\text{sp}\left\{
\begin{pmatrix}
M_{1}(y) & M_{1}(\bar{z}_1 y) & M_{1}(\bar{z}_2 y) \\ 
M_{1}(z_1 y) & M_{1}(|z_1|^2 y)  & M_{1}( \bar{z}_2 z_1 y) \\
M_{1}(z_2 y) & M_{1}( \bar{z}_1 z_2 y) & M_{1}(|z_2|^2 y)
\end{pmatrix} 
\right\}
= 
\begin{pmatrix}
  -27.0712\\
  -15.5635\\
   -9.5314\\
    \hphantom{-}7.1774\\
    \hphantom{-}9.9912\\
   \hphantom{-}18.8951\\
   \hphantom{-}19.0130\\
   \hphantom{-}27.9900\\
   \hphantom{-}45.6814
\end{pmatrix}
\end{equation}
where $\text{sp}\{\cdot \}$ stands for spectrum.
This shows that Condition 3 of Theorem \ref{th:trunc} does not hold, that is to say, $T_1$ and $T_2$ are not jointly hyponormal.
\end{example}

\begin{example}[Complex polynomial optimization over an ellipse]
\label{eg:ellipse}
\\\\
Minimize
$$ 3 - |z_1|^2 + \frac{1}{2} i\bar{z}_1 z_2^2 - \frac{1}{2} i\bar{z}_2^2z_1 + |z_2|^2 $$
over $z_1,z_2 \in \mathbb{C}$ subject to
$$  |z_1|^2 - \frac{1}{4} \bar{z}_1^2 - \frac{1}{4} z_1^2 - 1 = 0, ~~~  3 - |z_1|^2 - |z_2|^2 = 0, $$
$$ \overline{z}_2 - z_2 = 0, ~~~  \overline{z}_2 + z_2 \geqslant 0 .$$
The feasible set of this example is taken from \cite[Example 3.2]{josz-molzahn-2015}, itself originally from \cite{putinar-scheiderer-2012}. It is represented in Figure \ref{fig:ellipse}.
\begin{figure}[!h]
	\centering
	\includegraphics[width=.6\textwidth]{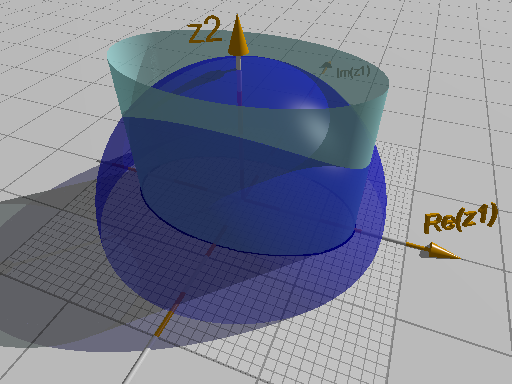}
	\caption{Feasible set in Example \ref{eg:ellipse}: intersection of semi-sphere and elliptic cylinder. We generated this figure using POV-Ray 3.7.0 \cite{povray2013}.}
	 \label{fig:ellipse}
\end{figure}

The first order complex relaxation is not defined because the variables appear to the second power (notice that $d_K = 2$). The second order relaxation yields the value 1.00047 and a spectral decomposition indicates that $\text{rank} M_0(y) = 1$ and $\text{rank} M_2(y) = 3$. Thus Condition 2 of Theorem \ref{th:trunc} does not hold. The third order relaxation yields 1.93291 and the following moment matrix:
$$
M_3(y) =  \hdots
$$
$$
\begin{array}{ccccccccccc}
 & 1 & z_1 & z_2 & z_1^2 & z_1 z_2 & z_2^2 & z_1^3 & z_1^2 z_2 & z_1 z_2^2 & z_2^3 \\[.25em]
1 & \hphantom{-}1.0000\hphantom{i} &  -0.1396i  & \hphantom{-}1.0193\hphantom{i} & \hphantom{-}1.9220\hphantom{i} &  -0.1423i  & \hphantom{-}1.0390\hphantom{i} & -0.8105i  & \hphantom{-}1.9591\hphantom{i} &  -0.1451i &  \hphantom{-}1.0590\hphantom{i}  \\[.25em]
\bar{z}_1 &   \hphantom{-}0.1396i &  \hphantom{-}1.9610\hphantom{i} & \hphantom{-}0.1423i  &  -0.2738i &  \hphantom{-}1.9989\hphantom{i} & \hphantom{-}0.1451i  & \hphantom{-}3.7691\hphantom{i} &  -0.2791i &  \hphantom{-}2.0375\hphantom{i} & \hphantom{-}0.1479i \\[.25em]
\bar{z}_2 &    \hphantom{-}1.0193\hphantom{i} &  -0.1423i &  \hphantom{-}1.0390\hphantom{i} & \hphantom{-}1.9591\hphantom{i} &  -0.1451i & \hphantom{-}1.0590\hphantom{i} & -0.8262i &  \hphantom{-}1.9970\hphantom{i} & -0.1479i  & \hphantom{-}1.0795\hphantom{i} \\[.25em]
\bar{z}_1^2 &    \hphantom{-}1.9220\hphantom{i} & \hphantom{-}0.2738i  & \hphantom{-}1.9591\hphantom{i} & \hphantom{-}3.8456\hphantom{i} & \hphantom{-}0.2791i & \hphantom{-}1.9970\hphantom{i} & -0.5369i &  \hphantom{-}3.9198\hphantom{i} & \hphantom{-}0.2845i  & \hphantom{-}2.0355\hphantom{i}  \\[.25em]
\bar{z}_1 \bar{z}_2 & \hphantom{-}0.1423i  & \hphantom{-}1.9989\hphantom{i} & \hphantom{-}0.1451i &  -0.2791i  & \hphantom{-}2.0375\hphantom{i} &  \hphantom{-}0.1479i &  \hphantom{-}3.8419\hphantom{i} &  -0.2845i  & \hphantom{-}2.0768\hphantom{i} & \hphantom{-}0.1507i \\[.25em]
\bar{z}_2^2 & \hphantom{-}1.0390\hphantom{i} & - 0.1451i  & \hphantom{-}1.0590\hphantom{i} & \hphantom{-}1.9970\hphantom{i} &  -0.1479i & \hphantom{-}1.0795\hphantom{i} &  -0.8421i  & \hphantom{-}2.0355\hphantom{i} &  -0.1507i &  \hphantom{-}1.1003\hphantom{i}  \\[.25em]
\bar{z}_1^3 & \hphantom{-}0.8105i &  \hphantom{-}3.7691\hphantom{i} & \hphantom{-}0.8262i  & \hphantom{-}0.5369i  & \hphantom{-}3.8419\hphantom{i} & \hphantom{-}0.8421i &  \hphantom{-}7.5412\hphantom{i} & \hphantom{-}0.5473i  & \hphantom{-}3.9161\hphantom{i} & -0.8584i \\[.25em]
\bar{z}_1^2 \bar{z}_2 &    \hphantom{-}1.9591\hphantom{i} & \hphantom{-}0.2791i  & \hphantom{-}1.9970\hphantom{i} &  \hphantom{-}3.9198\hphantom{i} & \hphantom{-}0.2845i & \hphantom{-}2.0355\hphantom{i} & -0.5473i &  \hphantom{-}3.9955\hphantom{i} & \hphantom{-}0.2900i  & \hphantom{-}2.0748\hphantom{i}  \\[.25em]
\bar{z}_1 \bar{z}_2^2 & \hphantom{-}0.1451i  & \hphantom{-}2.0375\hphantom{i} & \hphantom{-}0.1479i   &  -0.2845i  & \hphantom{-}2.0768\hphantom{i} & \hphantom{-}0.1507i  & \hphantom{-}3.9161\hphantom{i} & -0.2900i &  \hphantom{-}2.1169\hphantom{i} & \hphantom{-}0.1536i \\[.25em]
\bar{z}_2^3 &    \hphantom{-}1.0590\hphantom{i} & -0.1479i  & \hphantom{-}1.0795\hphantom{i}  & \hphantom{-}2.0355\hphantom{i} & -0.1507i & \hphantom{-}1.1003\hphantom{i} & -0.8584i  & \hphantom{-}2.0748\hphantom{i} & -0.1536i  & \hphantom{-}1.1216\hphantom{i}
\end{array}
$$
Using a spectral decomposition, we find that $\text{rank} M_1(y) = \text{rank} M_3(y) = 2$. In addition, we find the Cholesky decomposition $M_3(y) = X^* X$ where
$$
X = \hdots
$$
$$ 
\begin{array}{cccccccccc}
1 & z_1 & z_2 & z_1^2 & z_1 z_2 & z_2^2 & z_1^3 & z_1^2 z_2 & z_1 z_2^2 & z_2^3 \\[.25em]
\hphantom{-}1.0000\hphantom{i} & -0.1396i & \hphantom{-}1.0193\hphantom{i} & \hphantom{-}1.9220\hphantom{i} & -0.1423i & \hphantom{-}1.0390\hphantom{i} & -0.8105i & \hphantom{-}1.9591\hphantom{i} & -0.1451i & \hphantom{-}1.0590\hphantom{i} \\[.25em]
\hphantom{-}0.0000\hphantom{i} & \hphantom{-}1.3934\hphantom{i} & \hphantom{-}0.0000\hphantom{i} & -0.3891i & \hphantom{-}1.4203\hphantom{i} & \hphantom{-}0.0000\hphantom{i} & \hphantom{-}2.6238\hphantom{i} & -0.3966i & \hphantom{-}1.4477\hphantom{i} & \hphantom{-}0.0000\hphantom{i}
\end{array}
$$
for which the column basis indexed by $\{1,z_1\}$ is readily identified. We then obtain the shift operators
\begin{equation}
T_1 = 
\begin{pmatrix}
 -0.1396i &  \hphantom{-}1.3934\hphantom{i} \\
  \hphantom{-}1.3934\hphantom{i} & -0.1396i
\end{pmatrix}
~~~\text{and}~~~
T_2 =
\begin{pmatrix}
\hphantom{-}1.0193\hphantom{i} & \hphantom{-}0.0000\hphantom{i} \\
\hphantom{-}0.0000\hphantom{i} & \hphantom{-}1.0193\hphantom{i}
\end{pmatrix}.
\end{equation}
For example, the image by $T_1$ of the column of $X$ indexed by $z_1z_2$ is equal to the column of $X$ indexed by $z_1^2 z_2$:
\begin{equation}
T_1 \times
\begin{pmatrix}
-0.1423i \\
\hphantom{-}1.4203\hphantom{i}
\end{pmatrix}
=
\begin{pmatrix}
\hphantom{-}1.9591\hphantom{i} \\
- 0.3966i
\end{pmatrix}.
\end{equation}
Using an eigendecomposition, we find that
\begin{equation}
\text{sp}\left\{
\begin{pmatrix}
I        & T_1^* & T_2^* \\
T_1    & T_1^* T_1  & T_2^* T_1  \\
T_2    & T_1^* T_2  & T_2^* T_2
\end{pmatrix} 
\right\}
= 
\begin{pmatrix}
   -0.0000 \\
   -0.0000 \\
   \hphantom{-}0.0000\\
    \hphantom{-}0.0000\\
    \hphantom{-}4.0000\\
    \hphantom{-}4.0000
\end{pmatrix}
\end{equation}
and
\begin{equation}
\text{sp}\left\{
\begin{pmatrix}
M_{1}(y) & M_{1}(\bar{z}_1 y) & M_{1}(\bar{z}_2 y) \\ 
M_{1}(z_1 y) & M_{1}(|z_1|^2 y)  & M_{1}( \bar{z}_2 z_1 y) \\
M_{1}(z_2 y) & M_{1}( \bar{z}_1 z_2 y) & M_{1}(|z_2|^2 y)
\end{pmatrix} 
\right\}
= 
\begin{pmatrix}
   -0.0000\\
    \hphantom{-}0.0000\\
    \hphantom{-}0.0000\\
    \hphantom{-}0.0000\\
    \hphantom{-}0.0000\\
    \hphantom{-}0.0000\\
    \hphantom{-}0.0000\\
    \hphantom{-}7.9175\\
    \hphantom{-}8.0825
\end{pmatrix}.
\end{equation}
This implies that $T_1$ and $T_2$ are jointly hyponormal and all three conditions of Theorem \ref{th:trunc} hold. An atomic measure can be extracted from the truncated data. Due to hyponormality, we may diagonalize the shifts simultaneously. Taking a random linear combination, we have
\begin{equation}
   -0.1140 T_1 + 0.7979  T_2 =
\begin{pmatrix}
   \hphantom{-}0.8133 + 0.0159i  &~ -0.1588 + 0.0000i \\
  -0.1588 + 0.0000i &~ \hphantom{....}0.8133 + 0.0159i
\end{pmatrix}
= PDP^*
\end{equation}
with
\begin{equation}
P = 
\begin{pmatrix}
    \hphantom{-}0.7071\hphantom{i} &  \hphantom{-}0.7071\hphantom{i} \\
   -0.7071\hphantom{i} &  \hphantom{-}0.7071\hphantom{i} 
\end{pmatrix}
\end{equation}
and
\begin{equation}
D = \begin{pmatrix}
   0.9721 + 0.0159i & ~ 0.0000 + 0.0000i \\
   0.0000 + 0.0000i & ~ 0.6545 + 0.0159i
\end{pmatrix}.
\end{equation}
The first coordinate of the atoms is given by the diagonal of 
\begin{equation} P^*T_1P =
\begin{pmatrix}
   -1.3934 - 0.1396i &~ -0.0000 - 0.0000i \\
   \hphantom{-}0.0000 + 0.0000i &~ \hphantom{-}1.3934 - 0.1396i
\end{pmatrix}
\end{equation}
The second coordinate of the atoms is given by the diagonal of 
\begin{equation} P^*T_2P =
\begin{pmatrix}
  1.0193 - 0.0000i &~ 0.0000 + 0.0000i \\
  0.0000 - 0.0000i &~ 1.0193 + 0.0000i
\end{pmatrix}
\end{equation}
The atoms can now be read by looking at the first diagonal elements, then the second diagonal elements:
\begin{equation}
\begin{pmatrix}
  -1.3934 - 0.1396i \\
   \hphantom{-}1.0193 - 0.0000i
\end{pmatrix},
~~~ 
\begin{pmatrix}
   1.3934 - 0.1396i\\
   1.0193 + 0.0000i
\end{pmatrix}.
\end{equation}
With $x_1$ denoting the column of $X$ indexed by $1$, and with $p_1$ and $p_2$ denoting the first and second colums of $P$, the respective weights of the atoms are
\begin{equation}
\begin{array}{l}
|x_1^* p_1|^2 = |1.0000 \times 0.7071|^2 = 0.5000,\\
|x_1^* p_2|^2 = |1.0000 \times 0.7071|^2 = 0.5000.
\end{array}
\end{equation}
\end{example}
\begin{example}[Enforcing joint hyponormality in a convex relaxation]
\label{eg:ellipsebis}
\\\\
Minimize
$$ 3 - |z_1|^2 + \frac{1}{2} i\bar{z}_1 z_2^2 - \frac{1}{2} i\bar{z}_2^2z_1 $$
over $z_1,z_2 \in \mathbb{C}$ subject to
$$  |z_1|^2 - \frac{1}{4} \bar{z}_1^2 - \frac{1}{4} z_1^2 - 1 = 0, ~~~  3 - |z_1|^2 - |z_2|^2 = 0, $$
$$ \overline{z}_2 - z_2 = 0, ~~~  \overline{z}_2 + z_2 \geqslant 0 .$$
The feasible set of this example is the same as in Example \ref{eg:ellipse} and Figure \ref{fig:ellipse}.

The first order complex relaxation is not defined because the variables appear to the second power (notice that $d_K = 2$). The second order relaxation yields the value 0.155089 and the following moment matrix:
$$
M_2(y) = 
\begin{array}{ccccccc}
 & 1 & z_1 & z_2 & z_1^2 & z_1 z_2 & z_2^2 \\[.25em]
1 & \hphantom{-}1.0000\hphantom{i} & - 0.3747i  & \hphantom{-}0.8485\hphantom{i} & \hphantom{-}1.8272\hphantom{i} & - 0.5100i &  \hphantom{-}1.0864\hphantom{i}  \\[.25em]
\bar{z}_1 & \hphantom{-}0.3747i &  \hphantom{-}1.9136\hphantom{i} &  \hphantom{-}0.5100i  & - 0.1929i & \hphantom{-}1.0505\hphantom{i}  &  \hphantom{-}0.9313i \\[.25em]
\bar{z}_2 & \hphantom{-}0.8485\hphantom{i} & - 0.5100i  & \hphantom{-}1.0864\hphantom{i} &  \hphantom{-}0.9245\hphantom{i} & - 0.9313i  & \hphantom{-}1.4950\hphantom{i} \\[.25em]
\bar{z}_1^2 & \hphantom{-}1.8272\hphantom{i} & \hphantom{-}0.1929i &  \hphantom{-}0.9245\hphantom{i} & \hphantom{-}4.5886\hphantom{i} & \hphantom{-}0.1162i &  \hphantom{-}0.9324\hphantom{i} \\[.25em]
\bar{z}_1 \bar{z}_2 & \hphantom{-}0.5100i  & \hphantom{-}1.0505\hphantom{i} & \hphantom{-}0.9313i  & - 0.1162i  & \hphantom{-}1.1523\hphantom{i} & \hphantom{-}1.4140i \\[.25em]
\bar{z}_2^2 & \hphantom{-}1.0864\hphantom{i} & - 0.9313i &  \hphantom{-}1.4950\hphantom{i} & \hphantom{-}0.9324\hphantom{i} & - 1.4140i  & \hphantom{-}2.1069\hphantom{i}
\end{array}
$$
A spectral decomposition reveals that $\text{rank} M_1(y) = \text{rank} M_2(y) = 3$. In addition, we have the Cholesky decomposition $M_2(y) = X^* X$ where
$$
X ~~=~~
\begin{array}{cccccccccc}
1 & z_1 & z_2 & z_1^2 & z_1 z_2 & z_2^2 \\[.25em]
\hphantom{-}1.0000\hphantom{i} & -0.3747i  & \hphantom{-}0.8485\hphantom{i} & \hphantom{-}1.8272\hphantom{i} & -0.5100i & \hphantom{-}1.0864\hphantom{i} \\[.25em]
\hphantom{-}0.0000\hphantom{i} & -1.3316\hphantom{i} & -0.1442i & \hphantom{-}0.6591i & -0.6454\hphantom{i} & -0.3936i \\[.25em]
\hphantom{-}0.0000\hphantom{i} & \hphantom{-}0.0000\hphantom{i} & \hphantom{-}0.5879\hphantom{i} & -0.9030\hphantom{i} & -0.6896i & \hphantom{-}0.8785\hphantom{i}
\end{array}
$$
for which the column basis indexed by $\{1,z_1,z_2\}$ is readily identified. We then obtain the shift operators
\begin{equation}
T_1=
\begin{pmatrix}
 -0.3747i & -1.4777\hphantom{i} & -0.6893i \\[.5em]
  -1.3316\hphantom{i} & -0.1202i & \hphantom{-}0.8537\hphantom{i} \\[.5em]
   \hphantom{-}0.0000\hphantom{i} &  \hphantom{-}0.6781\hphantom{i} & -1.0067i \\[.5em]
\end{pmatrix}
\end{equation}
\begin{equation}
T_2 = 
\begin{pmatrix}
   \hphantom{-}0.8485\hphantom{i} &  \hphantom{-}0.1442i &  \hphantom{-}0.5879\hphantom{i} \\[.5em]
   \hphantom{-}0.1442i &  \hphantom{-}0.4441\hphantom{i} & -0.3525i \\[.5em]
   \hphantom{-}0.5879\hphantom{i} & \hphantom{-}0.3525i &  \hphantom{-}0.5593\hphantom{i} \\[.5em]
\end{pmatrix}
\end{equation}
Using an eigendecomposition, we find that
\begin{equation}
\text{sp}\left\{
\begin{pmatrix}
I        & T_1^* & T_2^* \\
T_1    & T_1^* T_1  & T_2^* T_1  \\
T_2    & T_1^* T_2  & T_2^* T_2
\end{pmatrix} 
\right\}
= 
\begin{pmatrix}
  -1.2759 \\
  -0.2532 \\
  -0.0000 \\
  -0.0000 \\
   \hphantom{-} 0.3735\\
    \hphantom{-}1.3206\\
    \hphantom{-}3.8963\\
    \hphantom{-}3.9388 \\
    \hphantom{-}4.0000\\
\end{pmatrix}
\end{equation}
and
\begin{equation}
\text{sp}\left\{
\begin{pmatrix}
M_{1}(y) & M_{1}(\bar{z}_1 y) & M_{1}(\bar{z}_2 y) \\ 
M_{1}(z_1 y) & M_{1}(|z_1|^2 y)  & M_{1}( \bar{z}_2 z_1 y) \\
M_{1}(z_2 y) & M_{1}( \bar{z}_1 z_2 y) & M_{1}(|z_2|^2 y)
\end{pmatrix}
\right\}
= 
\begin{pmatrix}
   -1.5874\\
   -0.1295\\
   -0.0000\\
    \hphantom{-}0.0000\\
    \hphantom{-}0.1574\\
    \hphantom{-}0.7711\\
    \hphantom{-}3.5471\\
    \hphantom{-}5.0544\\
    \hphantom{-}8.1869
\end{pmatrix} 
\end{equation}
This shows that Condition 3 of Theorem \ref{th:trunc} does not hold and that $T_1$ and $T_2$ are not hyponormal. Thus no measure can be extracted from the data.

The third order complex relaxation yields the value 0.428175 and the moment matrix satisfies $\text{rank}M_1(y) = \text{rank}M_3(y) = 1$. Thus a Dirac measure can be extracted, with support $(z_1,z_2) = (-0.8165i , 1.5275)$ and weight $1.0000$. Instead of computing the third order relaxation, we enforce hyponormality by adding the following constraint in the second order relaxation:
\begin{equation}
\begin{pmatrix}
M_{1}(y) & M_{1}(\bar{z}_1 y) & M_{1}(\bar{z}_2 y) \\ 
M_{1}(z_1 y) & M_{1}(|z_1|^2 y)  & M_{1}( \bar{z}_2 z_1 y) \\
M_{1}(z_2 y) & M_{1}( \bar{z}_1 z_2 y) & M_{1}(|z_2|^2 y)
\end{pmatrix} \succcurlyeq  0.
\end{equation}
We then obtain the value 0.428175 and the following moment matrix:
\begin{equation}
\label{eq:enforce}
M_2(y) = 
\begin{array}{ccccccc}
 & 1 & z_1 & z_2 & z_1^2 & z_1 z_2 & z_2^2 \\[.25em]
1 &    \hphantom{-}1.0000\hphantom{i} & - 0.8165i &  \hphantom{-}1.5275\hphantom{i} & -0.6667\hphantom{i}    & - 1.2472i &  \hphantom{-}2.3333\hphantom{i} \\[.25em]
\bar{z}_1 & \hphantom{-}0.8165i  & \hphantom{-}0.6667\hphantom{i} & \hphantom{-}1.2472i  & - 0.5443i   & \hphantom{-}1.0184\hphantom{i} & \hphantom{-}1.9052i\\[.25em]
\bar{z}_2 & \hphantom{-}1.5275\hphantom{i} & - 1.2472i &  \hphantom{-}2.3333\hphantom{i} & -1.0184\hphantom{i}  & - 1.9052i  & \hphantom{-}3.5642\hphantom{i} \\[.25em]
\bar{z}_1^2 & -0.6667\hphantom{i} & \hphantom{-}0.5443i & -1.0184\hphantom{i} & \hphantom{-}0.4444\hphantom{i}    & \hphantom{-}0.8315i & -1.5556\hphantom{i} \\[.25em]
\bar{z}_1 \bar{z}_2 &    \hphantom{-}1.2472i  & \hphantom{-}1.0184\hphantom{i} & \hphantom{-}1.9052i  & - 0.8315i   & \hphantom{-}1.5556\hphantom{i} &  \hphantom{-}2.9102i \\[.25em]
\bar{z}_2^2 &    \hphantom{-}2.3333\hphantom{i} & - 1.9052i  & \hphantom{-}3.5642\hphantom{i} &  -1.5556\hphantom{i}   & - 2.9102i &  \hphantom{-}5.4444\hphantom{i} 
\end{array}
\end{equation}
which satisfies $\text{rank}M_0(y) = \text{rank}M_2(y) = 1$. In addition, we have the Cholesky decomposition $M_2(y) = X^* X$ where
\begin{equation}
X ~~=~~
\begin{array}{cccccccccc}
1 & z_1 & z_2 & z_1^2 & z_1 z_2 & z_2^2 \\[.25em]
\hphantom{-}1.0000\hphantom{i} & - 0.8165i & \hphantom{-}1.5275\hphantom{i} & -0.6667\hphantom{i} & - 1.2472i & \hphantom{-}2.3333\hphantom{i}
\end{array}
\end{equation}
The shift operators can be read from $X$ and they are equal to
\begin{equation}
T_1 = \begin{pmatrix} - 0.8165i \end{pmatrix}
~~~\text{and}~~~
T_2 = \begin{pmatrix} 1.5275  \end{pmatrix}
\end{equation}
Using an eigendecomposition, we find that
\begin{equation}
\text{sp}\left\{
\begin{pmatrix}
I        & T_1^* & T_2^* \\
T_1    & T_1^* T_1  & T_2^* T_1  \\
T_2    & T_1^* T_2  & T_2^* T_2
\end{pmatrix} 
\right\}
= 
\begin{pmatrix}
   -0.0000\\
    \hphantom{-}0.0000\\
     \hphantom{-}4.0000
\end{pmatrix}
\end{equation}
and
\begin{equation}
\text{sp}\left\{
\begin{pmatrix}
M_{1}(y) & M_{1}(\bar{z}_1 y) & M_{1}(\bar{z}_2 y) \\ 
M_{1}(z_1 y) & M_{1}(|z_1|^2 y)  & M_{1}( \bar{z}_2 z_1 y) \\
M_{1}(z_2 y) & M_{1}( \bar{z}_1 z_2 y) & M_{1}(|z_2|^2 y)
\end{pmatrix}
\right\}
= 
\begin{pmatrix}
   -0.0000\\
    \hphantom{-}0.0000\\
    \hphantom{-}0.0000\\
    \hphantom{-}0.0000\\
    \hphantom{-}0.0000\\
    \hphantom{-}0.0000\\
    \hphantom{-}0.0000\\
    \hphantom{-}0.0000\\
   16.0000
\end{pmatrix} 
\end{equation}
Jointly hyponormality has been successfully enforced. It has reduced the rank of the moment matrix from rank 3 to rank 1. The eigenvalues of $T_1$ and $T_2$ are their values, thus we find the Dirac measure with support $(z_1,z_2) = (-0.8165i , 1.5275)$ and weight $1.0000$, precisely as in the third order relaxation. Notice that in the third order relaxation, there is one conic constraint of size $10 \times 10$ and one of size $6 \times 6$. In the second order relaxation augmented with \eqref{eq:enforce}, there is one conic constraint of size $9 \times 9$, one of size $6 \times 6$, and one of size $3 \times 3$.
\end{example}

\begin{example}[Complex polynomial optimization on the torus] 
\label{eg:toeplitz}
\\\\
Minimize
$$|z - e^{i\frac{\pi}{3}}|^2$$
over $z \in \mathbb{C}$ subject to 
$$ |z|^2 = 1 ~~~\text{and}~~~ z^3 = 1. $$
The feasible set is represented in Figure \ref{fig:unit}.
\begin{figure}[!h]
	\centering
	\includegraphics[width=.5\textwidth]{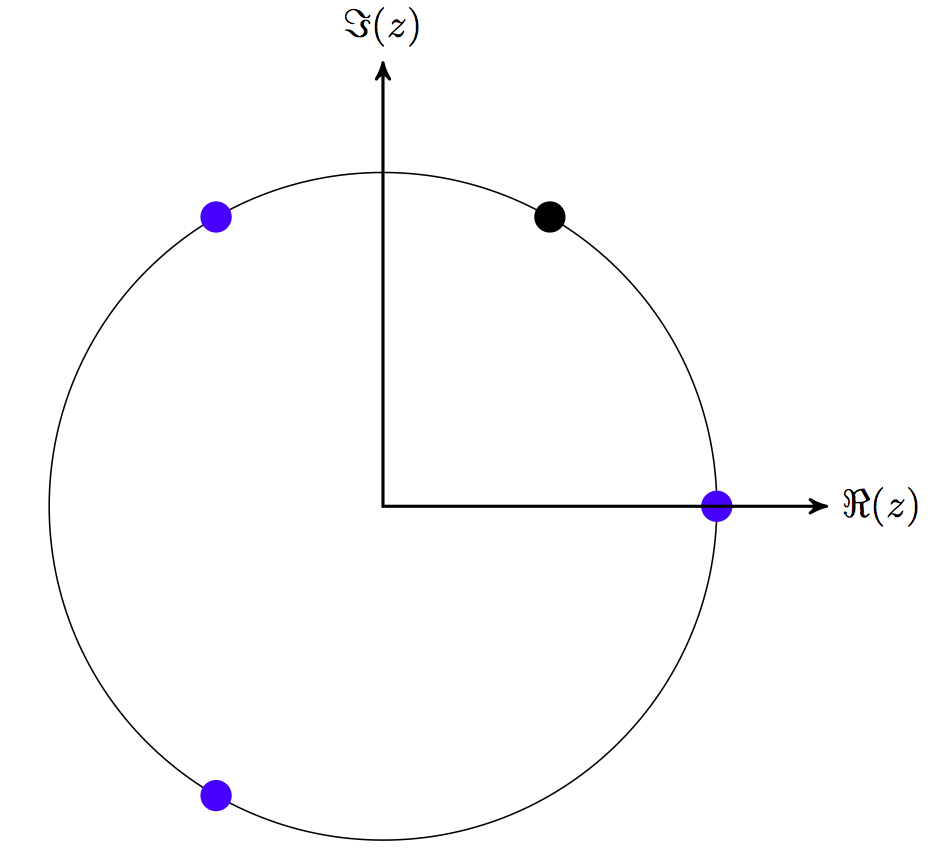}
	\caption{Feasible set in Example \ref{eg:toeplitz}: third roots of unity}
	 \label{fig:unit}
\end{figure}

The first and second order complex relaxations are not defined, while the third relaxation yields the value $0.999999$ and the following moment matrix:
$$
M_3(y) = \hdots
$$
$$
\begin{array}{ccccc}
 & 1 & z & z^2 & z^3 \\[.5em]
1 ~&  ~1.0000 - 0.0000i ~& ~ 0.2500 + 0.4330i ~ & ~ 0.2500 - 0.4330i ~&~  1.0000 - 0.0000i \\[.5em]
\bar{z} ~&   0.2500 - 0.4330i  & 1.0000 - 0.0000i  & 0.2500 + 0.4330i  & 0.2500 - 0.4330i \\[.5em]
\bar{z}^2 ~&   0.2500 + 0.4330i  & 0.2500 - 0.4330i &  1.0000 - 0.0000i &  0.2500 + 0.4330i \\[.5em]
\bar{z}^3 ~&   1.0000 - 0.0000i &  0.2500 + 0.4330i &  0.2500 - 0.4330i  & 1.0000 - 0.0000i
\end{array}
$$
Notice that the Toeplitz property holds. For instance, the term indexed by $(\bar{z},z^2)$ is equal to the term indexed by $(\bar{z}^2,z^3)$ (their common value is $0.2500 + 0.4330i$). The moment matrix satisfies $\text{rank}M_0(y) =1$ and $\text{rank}M_1(y) = \text{rank}M_2(y) = \text{rank}M_3(y) = 2$. While Condition 2 in Theorem \ref{th:toeplitz} does not hold since $d=d_K =3$, the rank is preserved between the current truncation order ($d=3$) and the previous ($d-1=2$). This suffices to guarantee the extraction of $\text{rank}M_{d-1}(y)$-atomic measure, i.e with 2 atoms. However, the atoms will not necessarily lie in the feasible set, so we must check for this in the end. We have the Cholesky decomposition $M_3(y) = X^*X$ where
$$
X = \hdots 
$$
$$
\begin{array}{cccc}
1 & z & z^2 & z^3  \\[.5em]
\hphantom{-}1.0000 + 0.0000i\hphantom{-} & \hphantom{-}0.2500 + 0.4330i\hphantom{-} & \hphantom{-}0.2500 - 0.4330i\hphantom{-} & \hphantom{-}1.0000 + 0.0000i\hphantom{-} \\[.5em]
\hphantom{-}0.0000 + 0.0000i\hphantom{-} & \hphantom{-}0.8660 + 0.0000i\hphantom{-} & \hphantom{-}0.4330 + 0.7500i\hphantom{-} & -0.0000 + 0.0000i\hphantom{-}
\end{array}
$$
for which the row basis indexed by $\{1,z\}$ is readily identified. We then obtain the shift operator
\begin{equation}
T = 
\begin{pmatrix}
   0.2500 + 0.4330i ~&~  0.4330 - 0.7500i \\
   0.8660 + 0.0000i ~&~ 0.2500 + 0.4330i
\end{pmatrix}
\end{equation}
which is unitary
\begin{equation}
T^*T = 
\begin{pmatrix}
   \hphantom{-}1.0000 + 0.0000i ~& -0.0000 - 0.0000i \\
  -0.0000 + 0.0000i  ~& \hphantom{-}1.0000 + 0.0000i
\end{pmatrix}.
\end{equation}
We have $T=PDP^*$ where
\begin{equation}
P = 
\begin{pmatrix}
  -0.6124 + 0.3536i ~&~  0.6124 - 0.3536i \\
  \hphantom{-}0.7071 + 0.0000i ~&~  0.7071 + 0.0000i
\end{pmatrix}
\end{equation}
and
\begin{equation}
D =  
\begin{pmatrix}
  -0.5000 + 0.8660i ~&~  0.0000 + 0.0000i \\
   \hphantom{-}0.0000 + 0.0000i ~&~  1.0000 + 0.0000i
\end{pmatrix}
\end{equation}
whence the two atoms are 
\begin{equation}
-0.5000 + 0.8660i ~~~\text{and}~~~ 1.0000 + 0.0000i
\end{equation}
Their respective weights are
\begin{equation}
\begin{array}{l}
|x_1^* p_1|^2 = |(1.0000 + 0.0000i)\times (-0.6124 + 0.3536i)|^2 = 0.5000,\\
|x_1^* p_2|^2 = |(1.0000 + 0.0000i)\times (\hphantom{-}0.7071 + 0.0000i)|^2 = 0.5000.
\end{array}
\end{equation}
It is easy to check that the atoms found are third roots of unity. In the fourth order relaxation, all conditions of Theorem \ref{th:toeplitz} hold, and we find the optimal value $1.0000$ and the same solutions.
\end{example}

\begin{example}[Real polynomial optimization] 
\label{eg:real}
\\\\
Minimize
$$-(x_1-1)^2 - (x_1-x_2)^2 - (x_2-3)^2$$
over $x_1,x_2,x_3 \in \mathbb{R}$ subject to 
$$ 1-(x_1-1)^2 \geqslant 0 , ~~~ 1-(x_1-x_2)^2 \geqslant 0 , ~~~ 1-(x_2-3)^2 \geqslant 0. $$
This example is taken from \cite[Ex. 5]{lasserre-2001}. Henrion and Lasserre's method for extracting global minimizers is illustrated on this example in \cite[Section 2.3]{henrion2005}. The feasible set is represented in Figure \ref{fig:triangle}.
\begin{figure}[!h]
	\centering
	\includegraphics[width=.5\textwidth]{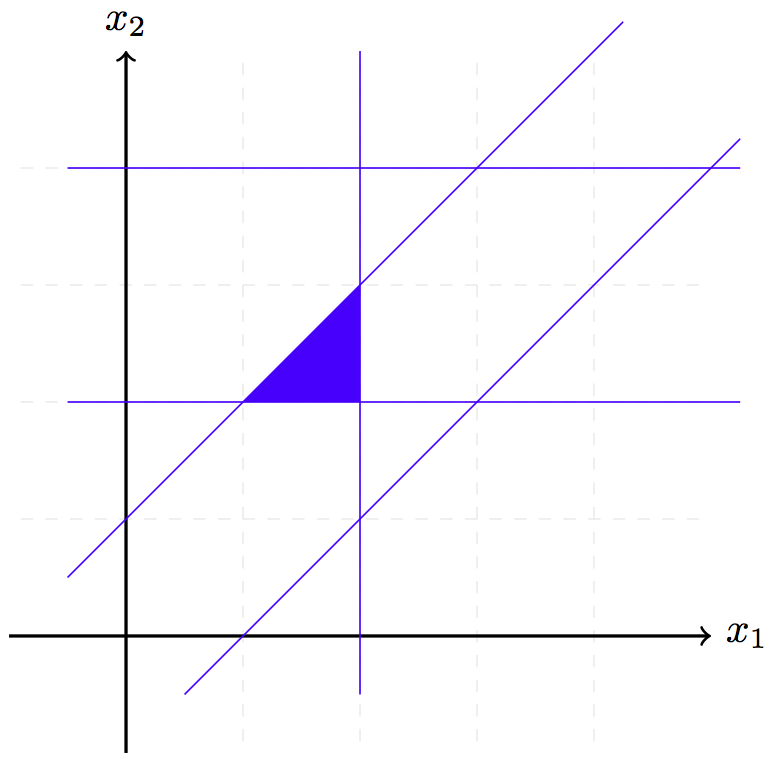}
	\caption{Feasible set in Example \ref{eg:real}: blue triangle}
	 \label{fig:triangle}
\end{figure}

The first order relaxation yields the value -3.0000 and a spectral decomposition indicates that $\text{rank} M_0(y) = 1$ and $\text{rank} M_1(y) = 3$. Thus, the rank is not preserved. The third order relaxation yields -2.0000 and the following moment matrix:
$$
M_2(y) = 
\begin{array}{ccccccc}
 & 1 & x_1 & x_2 & x_1^2 & x_1 x_2 & x_2^2 \\[.25em]
1 & 1.0000  &  1.4150 &   2.1182 &   2.2449  &  3.0663  & 4.5908  \\[.25em]
x_1 & 1.4150  &  2.2449  &  3.0663 &   3.9048  &  4.9625  &  6.8415 \\[.25em]
x_2 & 2.1182  &  3.0663  &  4.5908   & 4.9625   & 6.8415 &  10.2450 \\[.25em]
x_1^2 & 2.2449 &   3.9048 &   4.9625   & 7.2246  &  8.7549  & 11.3429\\[.25em]
x_1x_2 & 3.0663  &  4.9625  &  6.8415 &   8.7549  & 11.3429  & 15.8099\\[.25em]
x_2^2 & 4.5908  &  6.8415  & 10.2450  & 11.3429  & 15.8099  & 23.6804
\end{array}
$$
in which the rank is preserved since $\text{rank}M_1(y) = \text{rank}M_2(y) = 3$. Notice that the Hankel property holds. For instance the term indexed by $(x_1^2,x_2)$ is equal to the term indexed by $(x_1,x_1 x_2)$ (their common value is 4.9625).
We have the Cholesky decomposition $M_2(y) = X^TX$ where
$$
X = 
\begin{array}{cccccc}
1 & x_1 & x_2 &  x_1^2 &  x_1 x_2 & x_2^2 \\[.25em]
 \hphantom{-}1.0000 \hphantom{-} &  \hphantom{-}1.4150\hphantom{-} &  \hphantom{-}2.1182\hphantom{-} &  \hphantom{-}2.2449 \hphantom{-} &  \hphantom{-}3.0663\hphantom{-} &  \hphantom{-}4.5908\hphantom{-} \\[.25em]
0 & -0.4927\hphantom{-} & -0.1403\hphantom{-} & -1.4782 \hphantom{-} & -1.2660\hphantom{-} & -0.7015\hphantom{-} \\[.25em]
0 & 0 &  \hphantom{-}0.2907\hphantom{-} &  \hphantom{-}0.0000\hphantom{-} &  \hphantom{-}0.5814\hphantom{-} &  \hphantom{-}1.4536\hphantom{-}
\end{array}
$$
for which the row basis indexed by $\{1,x_1,x_2\}$ is readily identified. We then obtain the shift operators   
$$
T_1 = 
\begin{pmatrix}
    \hphantom{-}1.4150 &  -0.4927  & -0.0000 \\
   -0.4927   & \hphantom{-}1.5850 &  -0.0000 \\
         \hphantom{-}0.0000  & -0.0000 &   \hphantom{-}2.0000
\end{pmatrix}
~~~\text{and}~~~
T_2 =
\begin{pmatrix}
        \hphantom{-}2.1182  & -0.1403  &  \hphantom{-}0.2907 \\
   -0.1403  &  \hphantom{-}2.1666  & -0.3452 \\
    \hphantom{-}0.2907  & -0.3452 &   \hphantom{-}2.7152
\end{pmatrix}.
$$
which are symmetric matrices. For example, the image by $T_1$ of the column of $X$ indexed by $x_2$ is equal to the column of $X$ indexed by $x_1x_2$:
\begin{equation}
T_1 \times
\begin{pmatrix}
\hphantom{-}2.1182 \\
 -0.1403 \\
\hphantom{-}0.2907
\end{pmatrix}
=
\begin{pmatrix}
    \hphantom{-}3.0663" \\
   -1.2660 \\
    \hphantom{-}0.5814
\end{pmatrix}.
\end{equation}
Taking a random linear combination, we have    
$$  0.8494 T_1 + 0.1506 T_2 = 
\begin{pmatrix}
    \hphantom{-}1.5209  & -0.4396 &   \hphantom{-}0.0438 \\
   -0.4396  &  \hphantom{-}1.6726 &  -0.0520 \\
    \hphantom{-}0.0438  & -0.0520  &  \hphantom{-}2.1077
\end{pmatrix}  = PDP^T
$$
where
$$
P = 
\begin{pmatrix}
    \hphantom{-}0.7649 &   \hphantom{-}0.5448  &  \hphantom{-}0.3438 \\
    \hphantom{-}0.6442  & -0.6469 &  -0.4082 \\
   -0.0000  & -0.5336 &   \hphantom{-}0.8457
\end{pmatrix}   ,
~~~
D = 
\begin{pmatrix}
    1.1506      &   0     &    0 \\
    0   &  2.0000    &     0 \\
   0    &     0   &  2.1506
\end{pmatrix} .  
$$
The first coordinate of the atoms is given by the diagonal of 
$$
P^TT_1P =
\begin{pmatrix}
    \hphantom{-}1.0000  &  \hphantom{-}0.0000  & -0.0000 \\
    \hphantom{-}0.0000  &  \hphantom{-}2.0000  &  \hphantom{-}0.0000 \\
   -0.0000  &  \hphantom{-}0.0000 &   \hphantom{-}2.0000
\end{pmatrix}   
$$
while the second coordinate of the atoms is given by the diagonal of
$$
P^TT_2P =
\begin{pmatrix}
    \hphantom{-}2.0000  & -0.0000  &  \hphantom{-}0.0000 \\
   -0.0000  &  \hphantom{-}2.0000 &  -0.0000 \\
    \hphantom{-}0.0000  & -0.0000 &   \hphantom{-}3.0000
\end{pmatrix}.   
$$
The atoms can now be read
\begin{equation}
\begin{pmatrix}
    1.0000\\
    2.0000
\end{pmatrix}  ~~~ 
\begin{pmatrix}
    2.0000\\
    2.0000
\end{pmatrix}  ~~~ 
\begin{pmatrix}
    2.0000\\
    3.0000
\end{pmatrix}.
\end{equation}
With $x_1$ denoting the column of $X$ indexed by $1$, and with $p_1,p_2,p_3$ denoting the colums of $P$, the respective weights of the atoms are
\begin{equation}
\begin{array}{l}
(x_1^T p_1)^2 = |1.0000 \times 0.7649|^2 = 0.5850, \\[.2em]
(x_1^T p_2)^2 = |1.0000 \times 0.5448|^2 = 0.2968, \\[.2em]
(x_1^T p_3)^2 = |1.0000 \times 0.3438|^2 = 0.1182.
\end{array}
\end{equation}
\end{example}

\begin{example}[Prony's method via Autonne-Takagi factorization]
\label{eg:prony}

Consider the following complex exponential function where $z_1,z_2 \in \mathbb{C}$:
$$
f(z_1,z_2) := \hdots
$$
$$
1/4~ \exp(i\pi/2)~ \exp\{(-0.10+0.40i)z_1 + (0.05-0.80i)z_2\}
$$
$$ + $$
$$
1/3 ~ \exp(i4\pi/3) ~ \exp\{(0.03-0.35i)z_1 + (0.07-0.25i)z_2 \}.
$$
Its real part is represented in Figure \ref{fig:prony}.
\begin{figure}[!h]
	\centering
\includegraphics[width=.8\textwidth]{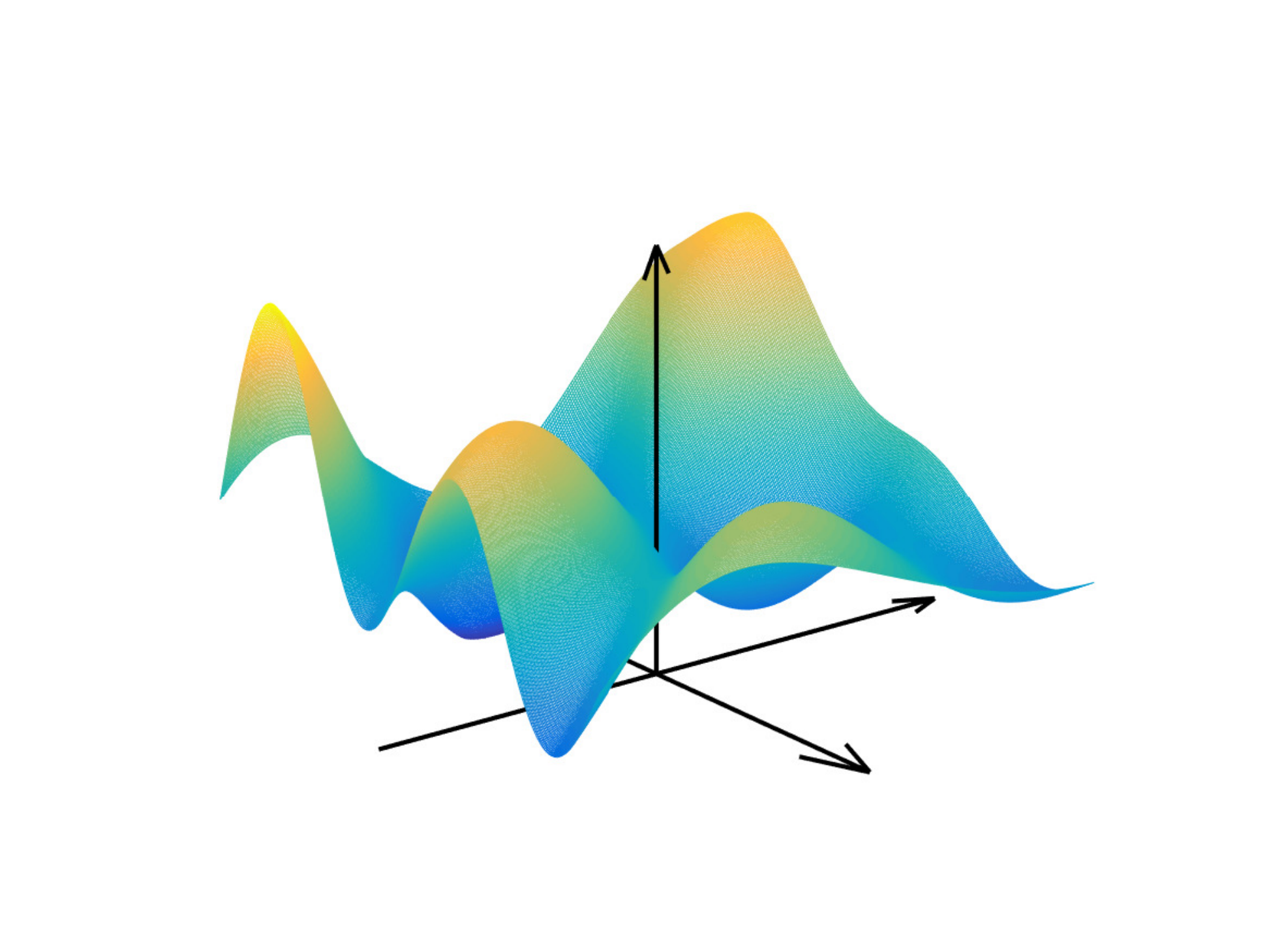}
	\caption{Signal to be reconstructed from evenly spaced measures (real part)}
	 \label{fig:prony}
\end{figure}

We reconstruct this function by evaluating it at $f(\alpha_1,\alpha_2)$ for $\alpha_1,\alpha_2 \in \mathbb{N}$ and $\alpha_1+\alpha_2 \leqslant 2d$. We start with $d=1$ and increment $d$ until the rank is preserved, i.e. $\text{rank}\mathcal{H}_{d}(y) = \text{rank}\mathcal{H}_{d-1}(y)$. We have $\text{rank}\mathcal{H}_0(y) = 1$ and $\text{rank}\mathcal{H}_1(y) = \text{rank}\mathcal{H}_2(y) = 2$. Thus we stop at the second order. The second order Hankel moment matrix contains the evaluations $f(\alpha_1,\alpha_2)$ for $\alpha_1+\alpha_2 \leqslant 4$. Precisely, $\mathcal{H}_d = ( f(\alpha+\beta) )_{|\alpha|,|\beta| \leqslant 2}$ and we have
$$
\mathcal{H}_2(y) = \hdots
$$
$$
\scriptsize
\begin{array}{ccccccc}
 & 1 & z_1 & z_2 & z_1^2 & z_1 z_2 & z_2^2 \\[.5em]
 1 & -0.1667 - 0.0387i & -0.3514 - 0.0122i & -0.0613 - 0.0727i & -0.4797 + 0.0222i & -0.2396 + 0.0597i & -0.0513 - 0.2076i\\[.5em]
z_1 &  -0.3514 - 0.0122i & -0.4797 + 0.0222i & -0.2396 + 0.0597i & -0.5373 + 0.0681i & -0.3778 + 0.1783i & -0.1544 + 0.0132i\\[.5em]
z_2 &  -0.0613 - 0.0727i & -0.2396 + 0.0597i & -0.0513 - 0.2076i & -0.3778 + 0.1783i & -0.1544 + 0.0132i & -0.1970 - 0.3346i\\[.5em]
z_1^2 &  -0.4797 + 0.0222i & -0.5373 + 0.0681i & -0.3778 + 0.1783i & -0.5202 + 0.1250i & -0.4546 + 0.2772i & -0.2401 + 0.2196i\\[.5em]
z_1z_2 &  -0.2396 + 0.0597i & -0.3778 + 0.1783i & -0.1544 + 0.0132i & -0.4546 + 0.2772i & -0.2401 + 0.2196i & -0.1842 - 0.0870i\\[.5em]
z_2^2 &  -0.0513 - 0.2076i & -0.1544 + 0.0132i & -0.1970 - 0.3346i & -0.2401 + 0.2196i & -0.1842 - 0.0870i & -0.4584 - 0.3256i
\end{array}
$$
The Autonne-Takagi factorization yields $\mathcal{H}_2(y) = X^TX$ where
$$
X = \hdots 
$$
$$
\scriptsize
\begin{array}{cccccc}
1 & z_1 & z_2 & z_1^2 & z_1 z_2 & z_2^2 \\[.5em]
-0.1052 + 0.4615i\hphantom{-} & \hphantom{-}0.0369 + 0.6628i\hphantom{-} & -0.0704 + 0.3889i\hphantom{-} & \hphantom{-}0.1866 + 0.7691i\hphantom{-} & \hphantom{-}0.1864 + 0.5507i\hphantom{-}  & -0.1207 + 0.3990i\hphantom{-}  \\[.5em]    
-0.2274 - 0.1285i\hphantom{-} & \hphantom{-}0.0626 - 0.2136i\hphantom{-} & -0.3707 + 0.2060i\hphantom{-} &  \hphantom{-}0.3184 - 0.2545i\hphantom{-} & -0.1736 - 0.0412i\hphantom{-} & -0.1935 + 0.5926i\hphantom{-}
\end{array}
$$
in which the row basis indexed by $\{1,z_1\}$ can be identified. We then obtain the complex symmetric shift operators
\begin{equation}
T_1 = 
\begin{pmatrix}
\hphantom{-}1.1490 - 0.3385i  &  -0.1879 - 0.3204i \\
-0.1879 - 0.3204i  &  \hphantom{-}0.6524 + 0.3376i
\end{pmatrix}
\end{equation}
\begin{equation}
T_2 =
\begin{pmatrix}
   0.9246 - 0.1751i & \hphantom{-}0.2857 + 0.0858i\\
  0.2857 + 0.0858i &  \hphantom{-}0.8470 - 0.8444i
\end{pmatrix}.
\end{equation}
For example, the image by $T_2$ of the column of $X$ indexed by $z_1$ is equal to the column of $X$ indexed by $z_1 z_2$:
\begin{equation}
T_2 \times 
\begin{pmatrix}
0.0369 + 0.6628i \\
0.0626 - 0.2136i
\end{pmatrix}
=
\begin{pmatrix}
\hphantom{-}0.1864 + 0.5507i \\
  -0.1736 - 0.0412i
\end{pmatrix}
\end{equation}
Taking a random linear combination, we get
$$   
0.8855 T_1  -0.1983 T_2 =
\begin{pmatrix}
   \hphantom{-}0.8341 - 0.2651i & -0.2230 - 0.3007i \\
  -0.2230 - 0.3007i & \hphantom{-}0.4098 + 0.4664i
\end{pmatrix} = PDP^T
$$
where
\begin{equation}
P = 
\begin{pmatrix}
1.0208 + 0.1231i & -0.3360 + 0.3740i \\
   0.3360 - 0.3740i & \hphantom{-}1.0208 + 0.1231i
\end{pmatrix}
\end{equation}
\begin{equation}
D =
\begin{pmatrix}
  0.6511 - 0.2603i &  \hphantom{.}0.0000 + 0.0000i \\
   0.0000 + 0.0000i &  \hphantom{.}0.5928 + 0.4616i
\end{pmatrix}
\end{equation}
We confirm that the inverse of $P$ is equal to its transpose since
\begin{equation}
P^T P = 
\begin{pmatrix}
   1.0000 + 0.0000i &  \hphantom{.}0.0000 - 0.0000i \\
   0.0000 - 0.0000i  & \hphantom{.}1.0000 + 0.0000i
\end{pmatrix}.
\end{equation}
The first coordinate of the atoms is given by the diagonal of
\begin{equation}
P^T T_1 P =
\begin{pmatrix}
 \hphantom{-}0.9680 - 0.3533i & \hphantom{.}0.0000 - 0.0000i \\
-0.0000 - 0.0000i & \hphantom{.}0.8334 + 0.3524i
\end{pmatrix}
\end{equation}
while the second coordinate of the atoms is given by the diagonal of
\begin{equation}
P^T T_2 P =
\begin{pmatrix}
   \hphantom{-}1.0392 - 0.2653i & -0.0000 - 0.0000i \\
  -0.0000 - 0.0000i &  \hphantom{-}0.7324 - 0.7541i
\end{pmatrix}.
\end{equation}
The atoms can now be read
\begin{equation}
\begin{pmatrix}
   0.9680 - 0.3533i \\
   1.0392 - 0.2653i
\end{pmatrix}
~~~\text{and}~~~
\begin{pmatrix}
   0.8334 + 0.3524i \\
   0.7324 - 0.7541i
\end{pmatrix}.
\end{equation}
With $x_1$ denoting the column of $X$ indexed by $1$, and with $p_1$ and $p_2$ denoting the first and second colums of $P$, the respective weights of the atoms are
\begin{equation}
\begin{array}{l}
(x_{0}^T p_1)^2 =  -0.1667 - 0.2887i , \\[.5em]
(x_{0}^T p_2)^2 = -0.0000 + 0.2500i. \\
\end{array}
\end{equation}
Taking the complex logarithm of the atoms, we obtain
\begin{equation}
\begin{pmatrix}
   0.0300 - 0.3500i \\
   0.0700 - 0.2500i
\end{pmatrix}
~~~\text{and}~~~
\begin{pmatrix}
  -0.1000 + 0.4000i \\
   \hphantom{-}0.0500 - 0.8000i 
\end{pmatrix}.
\end{equation}
As for the weights, they can be written
\begin{equation}
\begin{array}{l}
0.3333 \times \exp ( 4.1888i ) \\[.5em]
0.2500 \times \exp ( 1.5708i )
\end{array}
\end{equation}
which leads to the following function of $z_1$ and $z_2$
$$
0.2500~ \exp ( 1.5708i )~ \exp\{(-0.1000+0.4000i)z_1 + (0.0500-0.8000i)z_2\}
$$
$$ + $$
$$
0.3333 ~ \exp(4.1888i) ~ \exp\{(0.0300-0.3500i)z_1 + (0.0700-0.2500i)z_2 \}.
$$
The two terms in the sum can be visualized in Figure \ref{fig:prony12}.
\begin{figure}[!h]
\centering
\begin{minipage}{.5\textwidth}
  \centering
  \includegraphics[width=1.1\textwidth]{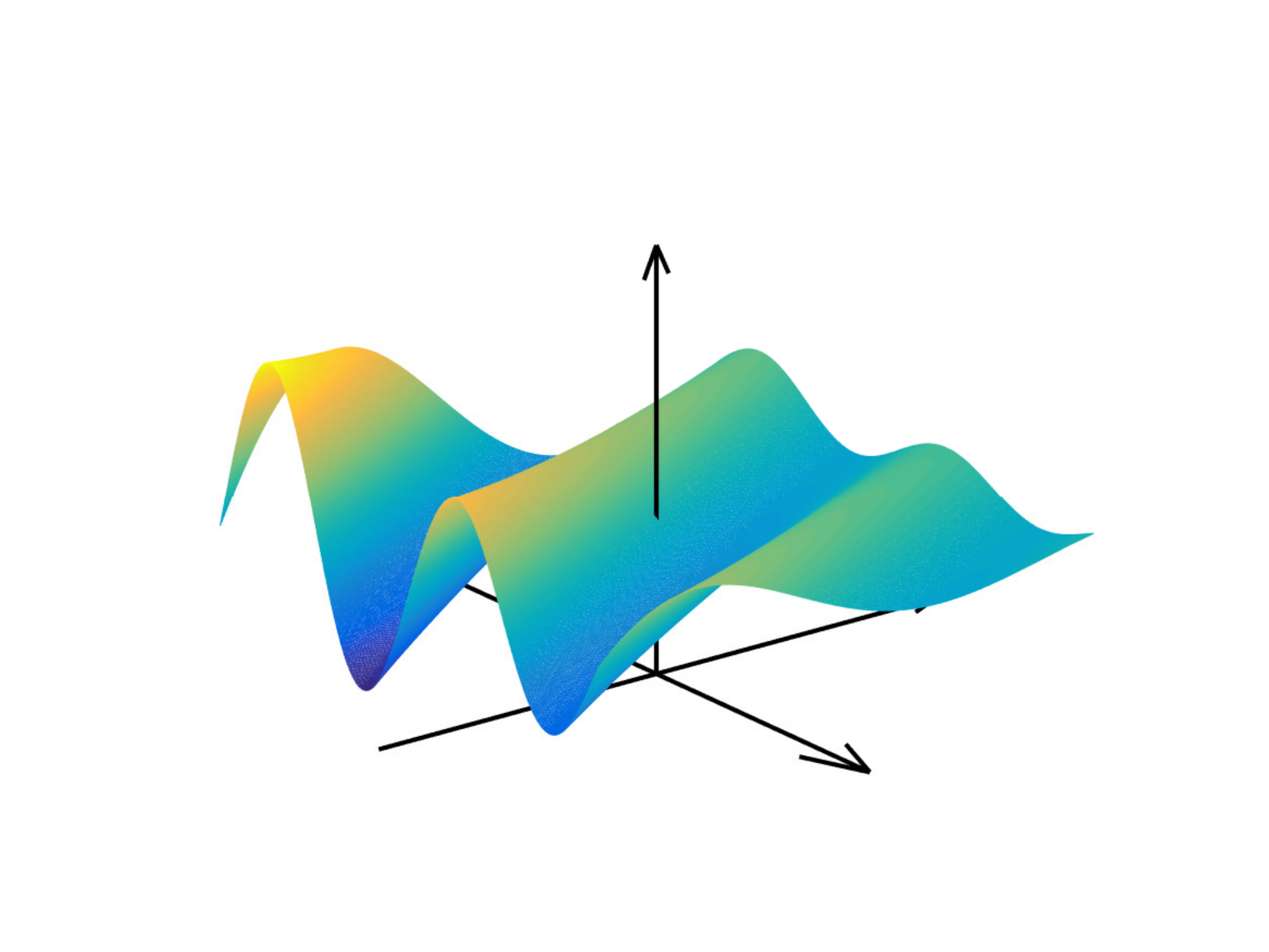}
  \label{fig:test1}
\end{minipage}%
\begin{minipage}{.5\textwidth}
  \centering
  \includegraphics[width=1.1\textwidth]{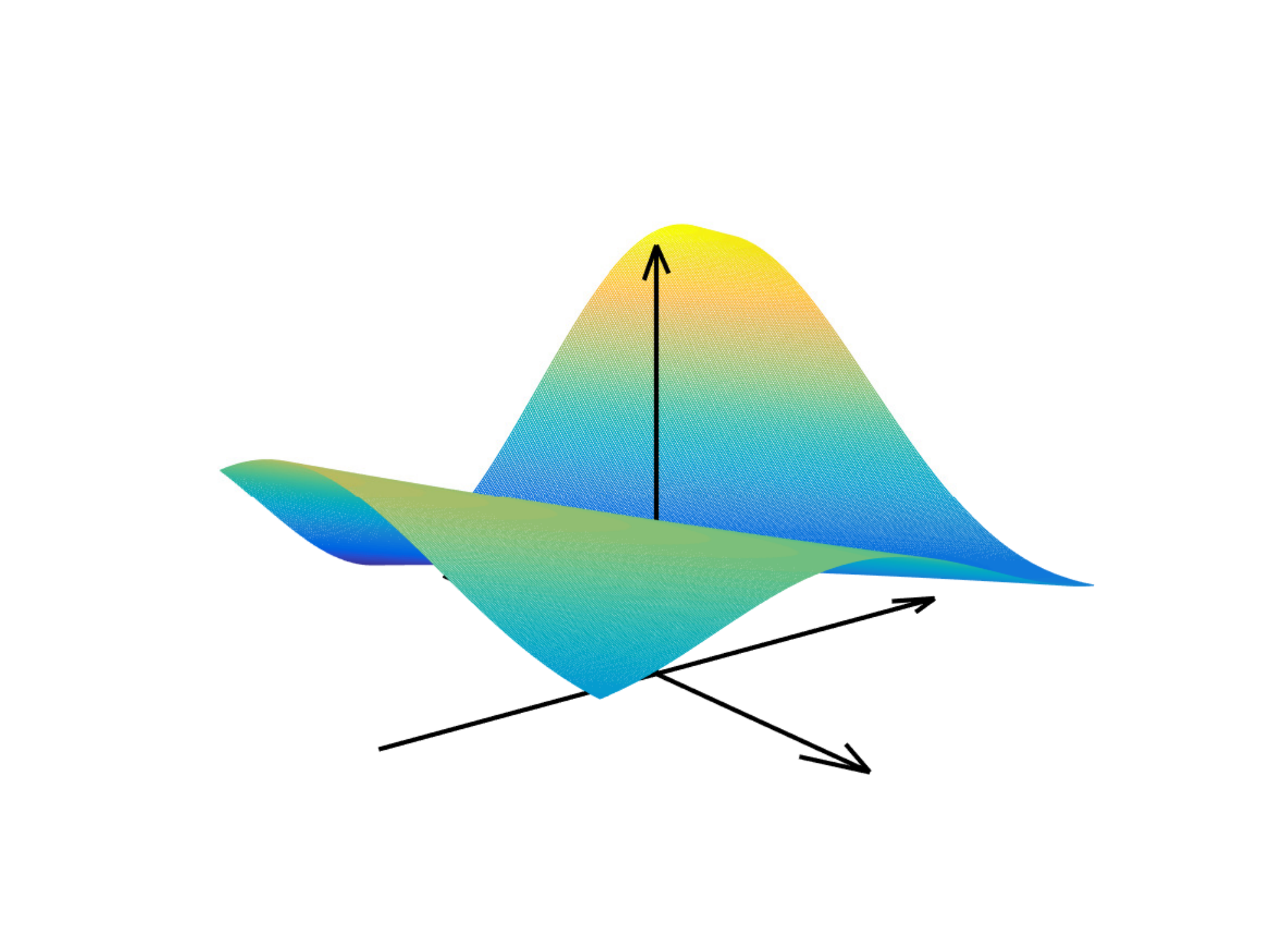}
  \label{fig:test2}
\end{minipage}
\caption{Decomposition of the signal into its components (real parts)}
\label{fig:prony12}
\end{figure}

\end{example}

%Having shown that the measure $\nu = \sum_{k=1} w_k \delta_{e^{[f_i]}}$ is uniquely determined by the measurements in Table \ref{tab:123}, we now discuss how to extract it from the measurements. Of course, there exists various approaches to do so.
%The constructive proof of the next result provides a new means of doing so. The resulting algorithm differs from Prony's method in that no Vandermonde system needs to be inverted in order to find the weights. We need to introduce new notation. 
%\begin{proposition}
%\label{prop:complex}
%\normalfont
%\textit{Given an integer} $d \neq 0$ \textit{and some complex numbers} $(y_\alpha)_{|\alpha|\leqslant 2d}$, \textit{there exists a unique rank-}$\mathcal{H}_{d-1}(y)$-\textit{atomic complex-valued measure such that}
%\begin{equation}
%y_{\alpha} = \int_{\mathbb{C}^n} z^\alpha d\mu~,~~~\forall |\alpha|\leqslant 2d
%\end{equation} 
%\textit{if and only if} $\text{rank}\mathcal{H}_d(y) = \text{rank}\mathcal{H}_{d-1}(y)$.
%\end{proposition}
%\begin{proof}
%The shifts exist and are complex symmetric thus $T_k = UD_kU^T$ with $D_k$ diagonal with complex entries and $U^TU = I$. Thus atoms may be extracted. For unicity, write equality between span then assume distinctness (proof by contradiction).
%\end{proof}

\section{Conclusion}
\label{sec:Conclusion}
An algorithm is proposed for finding global solutions to polynomial optimization problems and for exponential interpolation. It is founded on the notion of hyponormality and is related to the truncated moment problem. Various numerical applications are provided along with graphical illustrations.

\section*{Acknowledgements}
Many thanks to Paulo Ricardo Arantes Gilz, Didier Henrion, Jean Bernard Lasserre, and Mihai Putinar for fruitful discussions.

\pagebreak

\appendix

\section{First Lemma}
\begin{lemma} 
\label{lemma:ind}
If $z^{(1)}, \hdots , z^{(d)}$ are distinct points of $\mathbb{C}^n$, then $v_{d-1}(z^{(1)}), \hdots , v_{d-1}(z^{(d)})$ are linearly independent vectors, where $v_d(z) := (z^\alpha)_{|\alpha\leqslant d}$.
\end{lemma}
\begin{proof}
Consider some complex numbers $c_1, \hdots, c_d$ such that 
\begin{equation}
\label{eq:dependency}
\sum_{k=1}^d c_k (z^{(k)})^\alpha = 0 ~,~~~ \forall |\alpha|\leqslant d-1.
\end{equation}
Given $1\leqslant l \leqslant d$, define the Lagrange interpolation polynomial 
\begin{equation}
L^{(l)}(z) := \prod_{\scriptsize \begin{array}{c} 1 \leqslant k \leqslant n \\ k \neq l \end{array}}
\frac{ z_{i(k)} - z^{(k)}_{i(k)} }{ z^{(l)}_{i(k)} - z^{(k)}_{i(k)} } 
\end{equation}
 where $i(k) \in \{1,\hdots, n \}$ is an index such that $z^{(k)}_{i(k)} \neq z^{(l)}_{i(k)}$. It satisfies $L^{(l)}(z^{(k)}) = 1$ if $k = l$ and $L^{(l)}(z^{(k)}) = 0$ if $k \neq l$. The degree of $L^{(l)}(z) =: \sum_{\alpha} L^{(l)}_\alpha z^\alpha$ is equal to $d - 1$. Thus we may multiply the equation in \eqref{eq:dependency} by $L^{(l)}_\alpha$ to obtain
\begin{equation}
\sum_{k=1}^d c_k ~ L^{(l)}_\alpha (z^{(k)})^\alpha = 0 ~,~~~ \forall |\alpha|\leqslant d-1.
\end{equation}
Summing over all $|\alpha|\leqslant d-1$ yields $\sum_{k=1}^d c_k ~ L^{(l)}(z^{(k)}) = c_l = 0$. 
\end{proof}

\section{Second Lemma}
\begin{lemma}
\label{lemma:range}
\normalfont
\textit{If $u_1,\hdots,u_d \in \mathbb{C}^n$ are linearly independent, and $c_1,\hdots,c_d \in \mathbb{C}\setminus \{0\}$, then} $\mathcal{R}(\sum_{i=1}^d c_i u_i u_i^T) = \mathcal{R}(\sum_{i=1}^d c_i u_i u_i^*) = \text{span} \{u_1,\hdots,u_d\}$ where $\mathcal{R}$ denotes the range.
\end{lemma}
\begin{proof}
If $z \in \mathbb{C}^n$, then $(\sum_{i=1}^d c_i u_i u_i^T)z = \sum_{i=1}^d (c_i  u_i^T z) u_i \in \text{span} \{u_1,\hdots,u_d\}$ and $(\sum_{i=1}^d c_i u_i u_i^*)z = \sum_{i=1}^d (c_i  u_i^* z) u_i \in \text{span} \{u_1,\hdots,u_d\}$. Conversly, an element of the span $\sum_{i=1}^d \lambda_i u_i$ with $\lambda_1,\hdots,\lambda_n \in \mathbb{C}$ belongs to the the range of $\sum_{i=1}^d c_i u_i u_i^T$ if there exists $z\in\mathbb{C}^n$ such that 
$$ \sum_{i=1}^d \lambda_i u_i = \sum_{i=1}^d (c_i  u_i^T z) u_i $$ 
which is equivalent to each of the next three lines:
\begin{equation} \sum_{i=1}^d [ \lambda_i - (c_i  u_i^T z) ] u_i = 0, \end{equation}
\begin{equation}  \lambda_i = (c_i  u_i)^T z ~,~ i=1,\hdots,d,\end{equation}
\begin{equation} \lambda =  ( c_1 u_1 \hdots c_d u_d )^T z. \end{equation}
Since $( c_1 u_1 \hdots c_d u_d ) \in \mathbb{C}^{n \times d}$ has rank $d$, its transpose has rank $d$. Thus there exists a desired $z \in \mathbb{C}^n$. Likewise, $\sum_{i=1}^d \lambda_i u_i$ belongs to the the range of $\sum_{i=1}^d c_i u_i u_i^*$ if there exists $z\in\mathbb{C}^n$ such that 
$$  \lambda_i = (c_i  u_i)^* z ~,~ i=1,\hdots,d. $$
Since $( c_1 u_1 \hdots c_d u_d ) \in \mathbb{C}^{n \times p}$ has rank $d$, its conjugate transpose has rank $d$. Thus there exists a desired $z \in \mathbb{C}^n$.
\end{proof}

\bibliography{mybib}{}
\bibliographystyle{siam}

\end{document}